\documentclass[11pt]{article}

\usepackage[a4paper,margin=27mm]{geometry}
\usepackage[T1]{fontenc}
\usepackage{lmodern}
\usepackage{microtype}
\usepackage{amsmath,amssymb,amsthm,mathtools}
\numberwithin{equation}{section}
\usepackage{enumitem}
\usepackage{tikz}
\usepackage[hidelinks]{hyperref}

\newtheorem{theorem}{Theorem}[section]
\newtheorem{proposition}[theorem]{Proposition}
\newtheorem{lemma}[theorem]{Lemma}
\newtheorem{corollary}[theorem]{Corollary}
\theoremstyle{definition}

\theoremstyle{remark}

\newcommand{\E}{\mathbb E}
\newcommand{\Prb}{\mathbb P}
\newcommand{\Var}{\operatorname{Var}}

\setlist[itemize]{leftmargin=1.6em,itemsep=0.25em,topsep=0.35em}
\setlist[enumerate]{leftmargin=1.8em,itemsep=0.3em,topsep=0.4em}

\title{The terminal leave in random greedy hypergraph matching}

\author{Peter Keevash\thanks{Mathematical Institute, University of Oxford, UK. Supported by ERC Advanced Grant 883810.}}

\date{}

\begin{document}
\maketitle

\begin{abstract}
Let $G$ be the terminal graph of the random triangle-removal process
started from $K_n$.  We prove that $e(G)/n^{3/2}\to1/(2\sqrt2)$ in $L^2$.  
More generally, we obtain the analogous result for any fixed strictly $k$-balanced $k$-graph.
Our main tool is a more general theorem on random greedy matchings 
in nearly regular uniform hypergraphs with few short cycles
under a stability assumption on the survival equations.
\end{abstract}

\section{Introduction}
\label{sec:intro}

In the random triangle-removal process we start with the complete graph $K_n$, 
choose a uniformly random triangle from the current graph and
delete its edges, repeating until no triangle remains.  Let $G$
be the terminal graph.  We prove the sharp asymptotic
\[
 \frac{e(G)}{n^{3/2}}\xrightarrow{L^2}\frac1{2\sqrt2}.
\]
Bollob\'as recounts~\cite[pp.~231--232]{BollobasErdosHistory}
that he and Erd\H{o}s introduced this process
at the 1990 \emph{Quo Vadis, Graph Theory?} conference in
Fairbanks, conjecturing that the likely number of edges in the leave
has order $n^{3/2}$. Progress was obtained by Spencer~\cite{Spencer},
then Grable~\cite{Grable}, culminating in a breakthrough by
Bohman, Frieze and Lubetzky~\cite{BFL} who proved that
the terminal exponent is $3/2 + o(1)$. Joos and K\"uhn ~\cite{JoosKuhn} obtained
corresponding approximate exponent results for any strictly $k$-balanced $k$-graph.
Our method also gives the sharp asymptotic in this generality,
and even more generally for random greedy matchings
in nearly regular hypergaphs of high girth, 
under an additional stability assumption to be discussed below.

\subsubsection*{Statement on AI use and intellectual priority} 

I started sharing my ideas on the triangle removal process with ChatGPT in October 2025, having worked on it intermittently during the last decade, developing ideas that originated in my work with Tom Bohman \cite{bohman2010early,bohman2021dynamic}. I refined these ideas over many discussions during the past year, then over the last few weeks I developed the key breakthrough in solving the problem. The solution needed very little of the more sophisticated machinery that I developed during the past year (for which I expect to find other applications in the future). And remarkably, the solution (when specialised to triangles) was essentially the same as that in the recent OpenAI sharing progress in mathematics announcement (manuscript 188). The similarity and timing seem more than coincidental, so it is plausible that training on my conversations helped the OpenAI model solve the problem, although I have no direct evidence that this occurred. Recent online reactions by the mathematical community suggest that my situation is by no means unique. I would welcome clarification from OpenAI about whether their model used my work and how they propose to handle questions of intellectual priority and acknowledgement.

\subsubsection*{Contents and exposition of the paper}

In the next section I have tried to convey the main ideas and how they can naturally be discovered by a human mathematician rather than presented as an AI fait accompli. Thereafter the paper has two parts: the first proves the general result on random greedy matchings in hypergraphs; the second applies it to the hypergraph removal process for strictly $k$-balanced $k$-graphs. The main new ideas are in the first part, although the second part also requires ideas beyond those in the triangle case. The overall approach, structure and exposition of the paper reflect my mathematical choices, but some of the writing is AI-assisted, as following the recent OpenAI announcement, I am making my paper available promptly rather than delaying publication for stylistic improvements.

\section{Motivation and proof strategy}
\label{sec:motivation}

For concreteness, we describe the strategy as specialised to the triangle removal process.
We start in Subsection \ref{subsec:freeze} by setting the scene for the proof:
using the previously known results we can run the process to quite near the end, 
which we then `freeze' as a fixed graph satisfying various nice properties (the `handoff event').
The remaining analysis is local: for each surviving edge we analyse the probability that it survives
in the remaining graph. It is not as hard to show concentration, so this is the main difficulty.
Intuitively, as there are few triangles and they are randomly distributed, this survival probability 
should be well approximated by a tree heuristic. We explain this heuristic in Subsection \ref{subsec:tree} 
and describe an example showing why it can fail on arbitrary tree-like hypergraphs.
In Subsection \ref{subsec:stability} we explain why the tree prediction fails in this example
due to potential instability of the  survival equations. 
Then in Subsection \ref{subsec:actual} we state a stability assumption
that keeps this instability under control (but does not eliminate it entirely)
and indicate a coupling argument that allows us to conclude
an approximation for the actual process to the tree-like heuristic.
It remains to indicate in Subsection \ref{subsec:mixing} why the stability assumption
holds in the graph obtained from the triangle removal process handoff:
it can be deduced from known results on approximate counting of subgraphs,
specifically chains of triangles.

\subsection{Freezing the process and isolating the conditional mean} \label{subsec:freeze}

In the usual differential-equation picture of triangle removal
\cite{Wormald,BFL,JoosKuhn}, the current edge density is $p$, and a
surviving edge typically belongs to about $np^2$ triangles.  The
trajectory estimates are strongest while this number remains large.
Instead of extending them into the sparse terminal regime, stop after
a deterministic number $i_\star$ of deletions at density
\[
 p_\star:=1-\frac{6i_\star}{n^2}\sim n^{-1/2+\eta},
\]
where $\eta>0$ is a sufficiently small fixed constant.  Let $G_\star$
be the graph at that step (or the terminal graph if the process stops
earlier).  On the high-probability handoff event from Part~II, it has
about $N=n^2p_\star/2$ edges, each belonging to about
$D=np_\star^2$ triangles.

There is a useful way to describe the rest of the process.  Regard
the edges of $G_\star$ as \emph{resources}, and its triangles as
\emph{candidates}, each using its three edges.  Give the candidates
independent continuous priorities; inspect them in increasing
priority order, accepting a candidate exactly when none of its
resources has previously been used.  
For a candidate $A$, its acceptance can be decided without running
the whole process: consider the conflicting candidates of earlier
priority in increasing order, recursively decide whether each was
accepted, and stop as soon as one was.  This \emph{backwards query}
gives the correct answer by induction over the priority order.  It is very surprising that
this simple viewpoint suffices to analyse the remainder of the process.

Let $Z$ count the terminal graph edges and put
$\mu(G_\star)=\E[Z\mid G_\star]$.  Fluctuations about this
conditional mean are comparatively easy to control.  
Indeed, it follows from~\cite[Theorem~1.1]{DMM} that random greedy matching in
a simple $3$-uniform hypergraph with $N$ resources and maximum
resource degree $d$ leaves a number $Z_H$ of uncovered resources
satisfying $\Var Z_H\le C Nd^5$. On the good event for the triangle handoff we have
\[
 N=(1+o(1))\frac{n^2p_\star}{2},\qquad
 D=(1+o(1))np_\star^2.
\]
In particular, fluctuations are negligible on the $n^{3/2}$ scale: we have
\[
 \Var(Z\mid G_\star)\le Cn^7p_\star^{11}
 =O(n^{3/2+11\eta})=o(n^3).
\]
Thus the mean is the main difficulty. 
We will also estimate the variance without relying on \cite{DMM}.

\subsection{The regular-tree prediction and its limitations} \label{subsec:tree}

What terminal leave should we expect?
Suppose, as a first approximation, that a resource meets exactly $d$
candidates and that the branches reached through different candidates
are independent.  For a candidate branch, let $f_d(t)$ be the
probability that it does not consume its parent resource before priority
$t$.  It can consume that resource at time $t$ only if each of its
other $r-1$ resources has survived its remaining $d-1$ branches.
Thus $f_d'(t)=-f_d(t)^{(r-1)(d-1)}$, with $f_d(0)=1$.
The probability that the root survives is $q_d(t)=f_d(t)^d$,
so the same calculation gives
\[
 q_d'(t)=-d\,q_d(t)^{r(1-1/d)},\qquad q_d(0)=1.
\]
Setting $d=D$, which is large, we approximate by $q'=-D q^r$, with solution 
\[
 Q_D(t)=(1+(r-1)Dt)^{-1/(r-1)}.
\]
Consequently the predicted number of surviving resources is
\[
 NQ_D(1)
 \sim\frac{n^2p_\star/2}{\sqrt{2np_\star^2}}
 =\frac{n^{3/2}}{2\sqrt2}.
\]
The density $p_\star$ cancels, which explains why we can determine
the sharp constant without following the process into the very sparse regime.

Next we describe an example to show that we need more
from the handoff than being nearly regular and tree-like.
Consider three resource classes of sizes $a,a,b$, where
$a>b$, with every candidate using one resource from each class.  In
the complete tripartite example the resource degrees are $ab,ab,a^2$,
which are nearly equal when $a/b\to1$.  Yet every matching uses at
most $b$ candidates, leaving at least $2(a-b)$ resources.

This obstruction persists at arbitrarily large fixed incidence girth.
Indeed, fix $g\ge2$ and take $b=\lceil d^{4g}\rceil$ and
$a=b+\lceil bd^{-1/4}\rceil$.  Retain each tripartite triple
independently with probability $d/b^2$, then delete one candidate
from each incidence cycle of length at most $2g$.  Standard
concentration and first-moment estimates show that the degrees are
$d(1+O(d^{-1/4}))$ and that distinct short cycles have disjoint
full resource supports (including all resources belonging to their
candidates); thus each resource loses at most one candidate.  We
obtain a linear candidate hypergraph of incidence girth greater
than $2g$, with $N=2a+b=d^{O_g(1)}$ resources.  Every matching
still leaves at least $2(a-b)\asymp bd^{-1/4}$ resources, whereas
the regular-tree prediction $N/\sqrt{2d}$ is of order $bd^{-1/2}$.

This example is an abstract candidate hypergraph, not a triangle
system arising from a removal-process handoff.  It shows that
near-regularity, polynomial host size and arbitrarily large fixed
girth do not suffice: we must control how resource types interact
through candidates.

\subsection{Why the tree prediction can be unstable} \label{subsec:stability}

The tripartite example has a simple explanation in terms of the
 survival equations.  We ask whether survival probabilities
at different resources that start near the same reference curve remain
close throughout the continuation.

Write $q_v(x)$ for the survival probability suggested by the
independent-branch model at resource $v$.  In the leading approximation,
a candidate containing $v$ consumes it at priority $x$ if its other
$r-1$ resources survive to $x$.  This suggests the heuristic
\[
 \frac{d}{dx}\log q_v(x)
 \approx-\sum_{A\ni v}\prod_{u\in A\setminus\{v\}}q_u(x).
\]
We compare with the homogeneous profile $Q_D(x)$ from the preceding
subsection by writing
\[
 q_v(x)=Q_D(x)e^{y_v(\tau)},\qquad
 \tau=\frac{\log(1+(r-1)Dx)}{r-1}.
\]
Here $D$ is the typical resource degree.  Since
$Q_D'/Q_D=-DQ_D^{r-1}$ and $d\tau/dx=DQ_D^{r-1}$, the same
approximation gives
\[
 \frac{dy_v}{d\tau}
 \approx 1-\frac1D\sum_{A\ni v}
       \exp\!\left(\sum_{u\in A\setminus\{v\}}y_u\right).
\]
Expanding to first order in the relative deviations yields
\[
 \frac{dy_v}{d\tau}
 \approx 1-\frac{d_v}{D}
       -\frac1D\sum_{A\ni v}\sum_{u\in A\setminus\{v\}}y_u.
\]
The double sum suggests an operator on resources. 
Writing $M$ for the number of candidates, we define
\[
 (\mathsf T h)(v)
 =\frac1{(r-1)d_v}\sum_{A\ni v}
       \sum_{u\in A\setminus\{v\}}h(u),\qquad
 \pi(v)=\frac{d_v}{rM},
\]
This reversible Markov operator $\mathsf T$
chooses a uniform candidate through $v$ 
and then a different resource of that candidate.  
We then rewrite the heuristic as
\[
 \frac{dy_v}{d\tau}\approx
 1-\frac{d_v}{D}
 -(r-1)\frac{d_v}{D}(\mathsf T y)(v).
\]
This suggests $-(r-1)\mathsf T$ as the principal linear part when
the degrees are close to $D$, with $1-d_v/D$ as a forcing term.
If $\mathsf T h=\lambda h$ then $h$ evolves under the homogeneous equation as
$e^{-(r-1)\lambda\tau}h$.  Positive eigenvalues of $\mathsf T$
make deviations decay; negative ones make them grow.  
Initially every resource survives, so $q_v(0)=Q_D(0)=1$ and
$y(0)=0$.  A growing eigenmode cannot by itself create a
discrepancy: degree irregularity or another error must first
produce one.  For the exact equation with no linearisation,
we collect everything except the proposed linear part in a residual
$\mathcal E$.  The evolution can then be written as the identity
\[
 y'(\tau)=-(r-1)\mathsf T y(\tau)+\mathcal E(\tau),
 \qquad y(0)=0.
\]
Solving the linear part of this differential equation gives
\[
 y(\tau)=\int_0^\tau
 e^{-(r-1)(\tau-s)\mathsf T}\mathcal E(s)\,ds.
\]
Here $\mathcal E$ may itself depend on $y$; the formula shows how
its contributions are amplified, but does not by itself bound them.
The logarithmic clock runs until
\[
 T_D=\frac{\log(1+(r-1)D)}{r-1}=\Theta_r(\log D).
\]
An exponentially growing solution can therefore amplify a small
error by a power of $D$.  We must control this possible amplification
throughout the remaining interval.

The tripartite obstruction has exactly such an unstable mode.  Let
$h$ take values $1,1,-2$ on its three resource classes.  Every
candidate contains one resource from each class, so
$\sum_{u\in A}h(u)=0$ for every candidate $A$.  Consequently,
for a resource $v$,
\[
 (\mathsf T h)(v)
 =\frac1{2d_v}\sum_{A\ni v}
       \sum_{u\in A\setminus\{v\}}h(u)
 =-\frac12h(v).
\]
This remains true after the random construction and deletion of
short cycles, since every retained candidate is still tripartite.
Thus $e^{-2\tau\mathsf T}h=e^\tau h$.  At terminal time
$\tau=T_D=\tfrac12\log(1+2D)$ the amplification is
$\sqrt{1+2D}$: a small deviation created by forcing can be
magnified by a power of $D$ in the linearised dynamics.
To see whether this growth actually occurs, consider
\[
 \xi(\tau):=\frac{\langle y(\tau),h\rangle_\pi}
                       {\langle h,h\rangle_\pi},
 \qquad
 \varepsilon_h(\tau):=
     \frac{\langle\mathcal E(\tau),h\rangle_\pi}
          {\langle h,h\rangle_\pi}.
\]
As $\mathsf T$ is self-adjoint in $L^2(\pi)$ and
$\mathsf T h=-h/2$, projecting the equation for $y$ and solving gives
\[
 \xi(\tau)=\int_0^\tau e^{\tau-s}
                          \varepsilon_h(s)\,ds.
\]
So it is the forcing that first creates a component of $y$ in this
direction; the unstable evolution then amplifies it.  In particular,
$\xi'(0)=\varepsilon_h(0)$, since $\xi(0)=0$.

The degree imbalance supplies such forcing.  At time zero the exact
survival equations give $y'_v(0)=1-d_v/D$.  In the original complete
tripartite example, take $D=ab$; the degrees on the first two classes
are $ab$ and on the third are $a^2$.  Thus $y'(0)$ is zero on the first
two classes and equals $-(a-b)/b$ on the third.  Each class has
$\pi$-mass $1/3$, since every candidate contributes one incidence
to each class.  Consequently,
\[
\xi'(0)=
 \frac{\langle y'(0),h\rangle_\pi}
      {\langle h,h\rangle_\pi}
 =\frac{a-b}{3b}.
\]
A forcing constantly equal to this initial value
would give $\xi(\tau)=\frac{a-b}{3b}(e^\tau-1)$.
This does not show that the forcing persists,
but it diagnoses the failure of the proof,
while the tripartite example shows that
it is a failure of the conclusion, not just the proof.
 
\subsection{From stable tree predictions to the actual random matching} \label{subsec:actual}

Two questions remain.  Can we keep the independent-branch prediction
accurate until termination?  And, even if we can, why should it
predict the actual random matching, where repeated appearances of one
candidate use the same priority? 

\paragraph{Stability over the whole remaining time.}
The exact equations constructed in Section~\ref{sec:cover} have the
linear part derived above.  Section~\ref{sec:stability} proves that
the remaining terms consist of degree error
$\delta_D:=\max_v|d_v/D-1|$, a correction of order $\log D/D$,
and a quadratic error while deviations remain small.
Our stability assumption a bound, for an appropriate $q>2$,
on the operator norm $\|e^{tA}\|_{q\to q}$, which is
the largest amplification of the $L^q$ norm over all initial data.

We assume a bound of the form
\[
 \|e^{-(r-1)t\mathsf T}\|_{L^q(\pi)\to L^q(\pi)}
 \le C(1+t)^L.
\]
Writing $U(\tau)=\sup_{s\le\tau}\|y(s)\|_{L^q(\pi)}$,
the analysis above implies
\[
 U(\tau)\lesssim
 (1+\log D)^{L+1}
 \left(\delta_D+\frac{\log D}{D}
        +N^{1/q}U(\tau)^2\right).
\]
The factor $N^{1/q}$ arises when an $L^q$ estimate is converted
to a uniform bound over resource types.  A sufficiently large fixed
$q$ makes it affordable at the polynomial handoff, so the quadratic
term cannot drive the deviation out of the small regime.  
The approximation  will be proved first in $L^2$ and transferred to
$L^q$ by interpolation.  The mean survival probability in
the independent model satisfies
\[
 \bar q:=\frac1N\sum_v q_v(1)=(1+o(1))Q_D(1).
\]
Thus polynomial growth is enough; 
no decay estimate is needed.

\paragraph{Returning to the original process.}
We consider the backwards query from a resources. In the
actual hypergraph a candidate may be encountered along different
branches, with the same priority each time; the independent incidence
tree treats these as distinct occurrences with independent priorities.
We must show that such repetitions are rare.
We consider the first repeated candidate
and the unique path between its two occurences
in the unfolded  resource--candidate incidence tree.
If this path is short then  its image in the actual hypergraph 
is a rare short incidence cycle. If it is long and avoids short cycles
then the girth property gives a counting saving in the union bound,
which combines with the small probability of visiting such a path
to give the required bound.

\subsection{From triangle mixing to stability} \label{subsec:mixing}

It remains to justify the stability assumption.
The key extra input is quasirandomness of the triangles in $G$,
which we will express analytically in terms of mixing of link walks.
Fix a vertex $u$ of the handoff graph $G$, and consider its \emph{link}:
the graph on $N(u)$ in which $v,w$ are adjacent when $uvw$ is a
triangle of $G$.  Write $d_{uv}=|N(u)\cap N(v)|$.  The link walk and
its stationary law are
\[
 (K_ug)(v)=\frac1{d_{uv}}\sum_{w\in N(u)\cap N(v)}g(w),
 \qquad
 \sigma_u(v)=\frac{d_{uv}}{\sum_{x\in N(u)}d_{ux}}.
\]
Let $\Pi_u$ denote averaging with respect to $\sigma_u$.
The handoff gives us precise estimates on counts for chains of triangles,
which translate into estimates of the form
\[
 \sup_v\sum_{w\in N(u)}
 |K_u^\ell(v,w)-\sigma_u(w)|\le n^{-c_\eta}
\]
Since $K_u$ is reversible for $\sigma_u$, 
by duality $\|K_u^\ell-\Pi_u\|_{2\to2} \le n^{-c_\eta}$,
so  $\|K_u-\Pi_u\|_{2\to2} \le n^{-c_\eta/\ell}$.

Why is this useful for the resource walk?  Let $M$ be the number of
triangles, and give its resource edges the stationary law
$\pi(uv)=d_{uv}/(3M)$.  Starting from $uv$, the resource walk
chooses a triangle $uvw$ and moves to $uw$ or $vw$ with equal
probabilities.  For an edge function $f$, put $f_u(v)=f(uv)$.
Conditioning on the retained endpoint gives 
\[
 (\mathsf Tf)(uv)
 =\frac12\bigl((K_uf_u)(v)+(K_vf_v)(u)\bigr).
\]
We define the \emph{star-averaging operator}
by replacing link steps by their stationary law:
\[
 (\mathsf Sf)(uv)
 =\frac12\bigl((\Pi_uf_u)(v)+(\Pi_vf_v)(u)\bigr).
\]
The conditional distribution of $v$ given endpoint $u$ under the
resource law $\pi$ is exactly $\sigma_u$.  Jensen's inequality,
followed by summing in these conditional laws, therefore gives
\[
 \|\mathsf T-\mathsf S\|_{L^2(\pi)\to L^2(\pi)}
 \le n^{-c'_\eta}
\]
for some $c'_\eta>0$. As  $\mathsf T$ and $\mathsf S$ are Markov contractions 
on $L^\infty(\pi)$ we have $\|\mathsf T-\mathsf S\|_{\infty\to\infty} \le 2$,
so interpolation gives $q \to q$ bounds also with polynomial decay.
Choosing a large fixed $q$ makes the $N^{1/q}$ loss affordable 
while retaining a polynomially small operator error.

The stability problem has now been reduced to $\mathsf S$.
Indeed, suppose we can prove
$\|e^{-2t\mathsf S}\|_{q\to q}\le1+2t$.  Duhamel's formula,
with $\varepsilon_q$ as above, gives
\[
 e^{-2t\mathsf T}=e^{-2t\mathsf S}
 -2\int_0^t e^{-2(t-s)\mathsf S}
       (\mathsf T-\mathsf S)e^{-2s\mathsf T}\,ds.
\]
Taking norms and applying Gronwall's inequality yields
\[
 \|e^{-2t\mathsf T}\|_{q\to q}
 \le(1+2t)\exp\bigl(2\varepsilon_q(t+t^2)\bigr).
\]
The required time interval has length $O(\log D)=O(\log n)$,
so the exponential factor is $1+o(1)$.  It therefore suffices to
prove the claimed $1+2t$ bound for star averaging.  We have reduced
an analytic problem for the actual triangle walk to a simpler operator.

For that calculation, the star operation has a useful factorisation
on \emph{any}  graph $H$ equipped with positive edge weights
$\pi(e)$ summing to one.  We consider only vertices incident to
positive-weight edges.  Sample an edge with law $\pi$ and then
one of its endpoints uniformly.  The resulting vertex law is $\nu(v)=\frac12\sum_{e\ni v}\pi(e)$.
Define the two conditional averages on vertex and edge functions by
\[
 (Cg)(uv)=\frac{g(u)+g(v)}2,
 \qquad
 (C^*h)(v)=
 \frac{\sum_{e\ni v}\pi(e)h(e)}{\sum_{e\ni v}\pi(e)}.
\]
They are adjoint between $L^2(\nu)$ and $L^2(\pi)$, and each is an
$L^q$ contraction by Jensen's inequality.  The star operator is
exactly $\mathsf S=CC^*$.

To see why the negative evolution of $CC^*$ is stable, apply the
conditional averages in the opposite order.  From a vertex $v$,
choose an incident edge with probability proportional to its weight,
then choose one endpoint of that edge uniformly.  We return to $v$
with probability $1/2$; otherwise we make one step of the vertex
Markov walk
\[
 (Pg)(v)=
 \frac{\sum_{u:\,uv\in E(H)}\pi(uv)g(u)}
      {\sum_{u:\,uv\in E(H)}\pi(uv)}.
\]
Consequently
\[
2C^*C=I+P.
\]
Here $P$ is reversible for $\nu$, and, being Markov, its powers
are contractions on every $L^q(\nu)$.

Let $y(t)=e^{-2tCC^*}y_0$ for an arbitrary initial edge function
$y_0$, and set $z(t)=C^*y(t)$.  Then
\[
 z'(t)=-2C^*Cz(t)=-(I+P)z(t).
\]
This evolution for $z$ is closed, and crucially, it is contractive:
\[
 \|e^{-t(I+P)}\|_{q\to q}
 \le e^{-t}\sum_{j=0}^\infty\frac{t^j}{j!}
         \|P^j\|_{q\to q}\le1.
\]
Hence $\|z(t)\|_{L^q(\nu)}\le\|y_0\|_{L^q(\pi)}$.
Integrating $y'=-2Cz$ we thus justify the stability assumption
\[
 \|e^{-2tCC^*}\|_{L^q(\pi)\to L^q(\pi)}\le1+2t
 \qquad(t\ge0).
\]

\part{Random greedy matching}

In this part we state and prove the random greedy matching theorem motivated in Section~\ref{sec:motivation}. 
In Section \ref{sec:cover} we state the theorem, describe an independent model that captures the tree heuristic,
and set up a coupling argument for approximating the actual model by the independent model.
This reduces the proof of the theorem to two tasks: 
in Section \ref{sec:stability} we use the stability assumption to analyse the independent model,
then in Section \ref{sec:collision} we prove the bounds on collision probabilities needed to justify the coupling.

\section{Random greedy matching and independent covers} \label{sec:cover}

\subsection{The random greedy matching theorem}

We now state the random greedy matching theorem motivated in
Section~\ref{sec:motivation}.  Let $r\ge3$ be fixed and let $\Gamma$
be a finite simple $r$-uniform hypergraph, with resource vertices and
candidate edges.  Give candidates independent continuous priorities and
accept each in increasing order when it is disjoint from previously
accepted candidates.  Write $Z_\Gamma$ for the number of uncovered
resources, and put
\[
 N=|V(\Gamma)|,\qquad M=|E(\Gamma)|.
\]
For the resource degrees $d_v$, define the degree-weighted law and the
resource walk from Section~\ref{sec:motivation} by
\[
 \pi(v)=\frac{d_v}{rM},\qquad
 (\mathsf Th)(v)=\frac1{(r-1)d_v}
       \sum_{A\ni v}\sum_{u\in A\setminus\{v\}}h(u).
\]
These expressions are well defined when all resource degrees are
positive, as ensured by the near-regularity hypothesis below.
For $R\ge4$, let $\mathcal I(\Gamma)$ be the bipartite incidence graph
and set
\[
 \chi_R(\Gamma):=\frac1M
 |\{A\in E(\Gamma): A\text{ belongs to a simple incidence
 cycle of length at most }R\}|.
\]
Thus $\chi_R$ counts candidates lying on short cycles, not merely
candidates near them.

\begin{theorem}[Random greedy matching]
\label{thm:fixed-host}
With the notation above, for every fixed $r\ge3$ there are constants
$A_r,B_r,R_r<\infty$ with the following property.
Let $\Gamma=\Gamma_n$ be a sequence of simple $r$-uniform
hypergraphs and let $D=D_n\to\infty$.  Put
\[
 \delta_n:=\max_{v\in V(\Gamma)}
       \left|\frac{d_v}{D}-1\right|,
 \qquad
 T_D:=\frac{\log(1+(r-1)D)}{r-1}.
\]
Suppose that for fixed $q>2$ and $C,L,K<\infty$ the following
conditions hold:
\begin{enumerate}[label=\textup{(\roman*)}]
\item \emph{Size and stability:}
\[
 N\le D^K,
 \qquad
 \|e^{-(r-1)t\mathsf T}\|_{L^q(\pi)\to L^q(\pi)}
 \le C(1+t)^L\quad(0\le t\le T_D).
\]
\item \emph{Near-regularity and small stability error:}
$\delta_n\to0$ and
\[
 N^{1/q}\left(\delta_n+\frac{\log D}{D}\right)
 (1+\log D)^{2L+2}=o(1).
\]
\item \emph{Small collision errors:}
\[
 D^{A_r}\chi_{R_r}(\Gamma)+\frac{D^{B_r}}{M}
 =o\!\left(D^{-2/(r-1)}\right).
\]
\end{enumerate}
Then, writing $Q_D(1)=(1+(r-1)D)^{-1/(r-1)}$, we have
\[
 \E\left[\left(\frac{Z_\Gamma}{NQ_D(1)}-1\right)^2\right]=o(1).
\]
In particular, $Z_\Gamma/(NQ_D(1))\to1$ in probability.
\end{theorem}

The proof has two distinct tasks.  We first approximate the survival
probability of each resource by solving a model with independent
candidate occurrences.  We must then justify replacing the actual
actual random matching by that independent model, whose different occurrences
of the same  candidate have different priorities.  This section
constructs both backwards queries, derives the independent-cover
equations, and proves a coupling that reduces the first and second
moments to these two tasks.  Section~\ref{sec:stability} estimates the
independent survival law, while Section~\ref{sec:collision} controls
repeated candidate types in finite-query explorations.

\subsection{The exact finite backwards test}

Fix a resource $v\in V(\Gamma)$, and let $X_v$ indicate that $v$
is uncovered at the end of greedy matching with independent uniform
candidate priorities $(U_A:A\in E(\Gamma))$.  We first construct an
exact backwards query rooted at $v$ to determine $X_v$.  A resource is
\emph{free before $t$} if no accepted candidate with priority less than
$t$ contains it.  The root query asks whether $v$ is free before $1$.
To define it recursively, we allow a call with a general focus set $S$
and threshold $t$, asking whether every resource of $S$ is free before
$t$.  An incoming candidate $P$ is omitted from that call's scanner,
but not from those of its descendants.

Fix total orders on $V(\Gamma)$ and $E(\Gamma)$ and order the 
resource--candidate incidences $(u,A)$, $u\in A$, lexicographically by
resource and then candidate.  For a nonempty $S\subseteq V(\Gamma)$,
$t\in[0,1]$ and
$P\in E(\Gamma)\cup\{\bot\}$, define $\mathsf{Test}(S,t;P)$, where
$P=\bot$ at the root and otherwise $U_P=t$.  Its scanner consists of the
incidences
\[
 \mathcal L(S;P)
 :=\{(u,A):u\in S,\ A\in E(\Gamma),\ u\in A,\ A\ne P\}.
\]
The scanner is a list of incidences, not of distinct candidates.  A
candidate meeting several resources of $S$ contributes one entry for each
such resource.  The call follows three steps:
\begin{enumerate}[label=\textup{(\arabic*)}]
\item Reveal the priority of every scanner entry in the fixed incidence
order, even when the priority is above $t$.
\item Sort the entries by priority, breaking ties by the fixed incidence
order, and consider only entries of priority below $t$.
\item For each remaining $(u,A)$ in that order, make the \emph{single}
recursive call
\[
 \mathsf{Test}(A\setminus\{u\},U_A;A).
\]
Return false at the first true recursive call, and true if all such
calls return false.
\end{enumerate}
Distinct candidates have different priorities almost surely, but a
candidate $A$ meeting $S$ twice gives two entries with the \emph{same}
priority $U_A$ and different recursive focus sets.  The fixed tie-break makes
the evaluation unambiguous, without affecting its Boolean answer.

The same candidate retains its priority $U_A$ if it reappears in
another branch.  Every recursive threshold decreases, so the test terminates.

\begin{lemma}[Correctness of the finite query]
\label{lem:finite-query}
Every call $\mathsf{Test}(S,t;P)$, with $P=\bot$ or $U_P=t$,
returns true exactly when all resources in $S$ are free before $t$.
In particular, $X_v=\mathsf{Test}(\{v\},1;\bot)$.
\end{lemma}

\begin{proof}
Induct on the number of  candidate priorities below $t$;
every recursive call has a smaller threshold.  If all resources in
$S$ are free before $t$, a true recursive call at an entry $(u,A)$
with $U_A<t$ would say, by induction, that the other resources of $A$
are free before $U_A$.  Since $u$ is free as well, $A$ would have been
accepted, a contradiction.  Thus the test returns true.

Conversely, if a resource of $S$ is covered before $t$, choose an
accepted candidate $A$ meeting $S$ with $U_A<t$ and $u\in A\cap S$.
The entry $(u,A)$ is in the scanner: an omitted incoming candidate
has priority $t$, not less than $t$.  All resources of
$A\setminus\{u\}$ were free before $U_A$, so induction says that
this entry's recursive call returns true.  The test therefore returns
false.  The argument also covers repeated entries of one candidate
in a merged scanner.
\end{proof}

\subsection{Independent occurrences and the typed incidence cover}

We now construct a second query, rooted at the \emph{same} resource
type $v$.  Unfold the resource--candidate incidence graph of $\Gamma$
into the bipartite \emph{typed incidence cover} $\mathcal T_v$,
rooted at a resource occurrence of type $v$.  A resource occurrence of
type $u$ has one candidate child occurrence of type $A$ for every
$A\ni u$, apart from its incoming candidate.  A candidate occurrence
of type $A$ reached through a resource of type $u$ has $r-1$ forward
resource children, of the distinct types in $A\setminus\{u\}$.
Occurrences are distinct even when their  types coincide.

Give all candidate occurrences independent uniform priorities and run
the same backwards query from the root at threshold $1$.  In
particular, invoking an occurrence of type $A$ through type $u$
produces \emph{one} recursive test of its $r-1$ forward resources,
using their merged scanner.  Let $Y_v$ be the answer.  Thus $X_v$
and $Y_v$ concern the same  resource type, but $X_v$ uses one
priority $U_A$ for every appearance of  candidate $A$, whereas
distinct cover occurrences have independent priorities.

A candidate occurrence is \emph{exposed} when a scanner reads its
priority and \emph{invoked} when its recursive test is evaluated.
Occurrences whose priorities exceed a threshold, or whose tests are
skipped after an earlier true answer, are exposed but not invoked.
The forward subtrees of distinct cover occurrences are disjoint.

We use the fixed resource and candidate orders from the finite query.
Whenever an invoked candidate produces its forward resource
occurrences, register all their births in resource-type order before
exposing any scanner entries.  Expose those entries in
resource-then-candidate order, before making recursive calls in increasing
priority order (with incidence order breaking ties).  The same
convention applies to the finite query.  These choices determine the
order of births and exposures, without changing the Boolean answers.

For example, take $r=3$ and candidates $A=\{u,v,w\}$ and
$B=\{v,w,x\}$.  If $A$ is invoked through $u$, its merged scanner
contains both $(v,B)$ and $(w,B)$.  In the finite query these entries
have the same priority $U_B$ but different recursive tests.  Both entries
are exposed even if $U_B>U_A$, when neither child is invoked.  In the
independent cover they are two formal candidate occurrences with fresh,
independent priorities, both labelled by the  type $B$.
Figure~\ref{fig:scanner-collision} shows the two exposed entries and the
 incidence cycle $A,v,B,w,A$.

\begin{figure}[tb]
\centering
\begin{minipage}[c]{.47\linewidth}
\centering
\begin{tikzpicture}[x=.9cm,y=.77cm,
 resource/.style={circle,draw,inner sep=1.3pt,minimum size=6mm},
 candidate/.style={rectangle,draw,rounded corners=1pt,inner sep=3pt,minimum width=8mm}]
 \node[resource] (u) at (0,2.6) {$u$};
 \node[candidate] (a) at (0,6.6) {$A$};
 \node[resource] (v) at (-1.4,1.35) {$v$};
 \node[resource] (w) at (1.4,1.35) {$w$};
 \node[candidate] (b) at (0,.15) {$B$};
 \node[resource] (x) at (0,-.85) {$x$};
 \draw (u)--(a)--(v)--(b)--(w)--(a);
 \draw (b)--(x);
\end{tikzpicture}

\small Finite incidence graph
\end{minipage}\hfill
\begin{minipage}[c]{.47\linewidth}
\centering
\begin{tikzpicture}[x=.9cm,y=.77cm,
 resource/.style={circle,draw,inner sep=1.3pt,minimum size=6mm},
 candidate/.style={rectangle,draw,rounded corners=1pt,inner sep=3pt,minimum width=8mm}]
 \node[resource] (u) at (0,2.6) {$u$};
 \node[candidate] (a) at (0,6.6) {$A$};
 \node[resource] (v) at (-1.4,1.35) {$v$};
 \node[resource] (w) at (1.4,1.35) {$w$};
 \node[candidate] (bv) at (-1.4,.15) {$B^{(v)}$};
 \node[candidate] (bw) at (1.4,.15) {$B^{(w)}$};
 \draw (u)--(a)--(v);
 \draw (a)--(w);
 \draw[dashed] (v)--(bv);
 \draw[dashed] (w)--(bw);
 \node[font=\scriptsize,align=center] at (0,-.68) {distinct occurrences,\\same  type $B$};
\end{tikzpicture}

\small Independent incidence cover
\end{minipage}
\caption{The same  candidate in two entries of a merged scanner.
On the left, $A,v,B,w,A$ is a four-cycle in the finite incidence graph.
On the right, the solid edges belong to the invoked query through $A$;
dashed edges denote the two candidate occurrences exposed in its forward
scanners, which need not be invoked.  Their  type is $B$ in both
cases, but their cover priorities are independent.}
\label{fig:scanner-collision}
\end{figure}
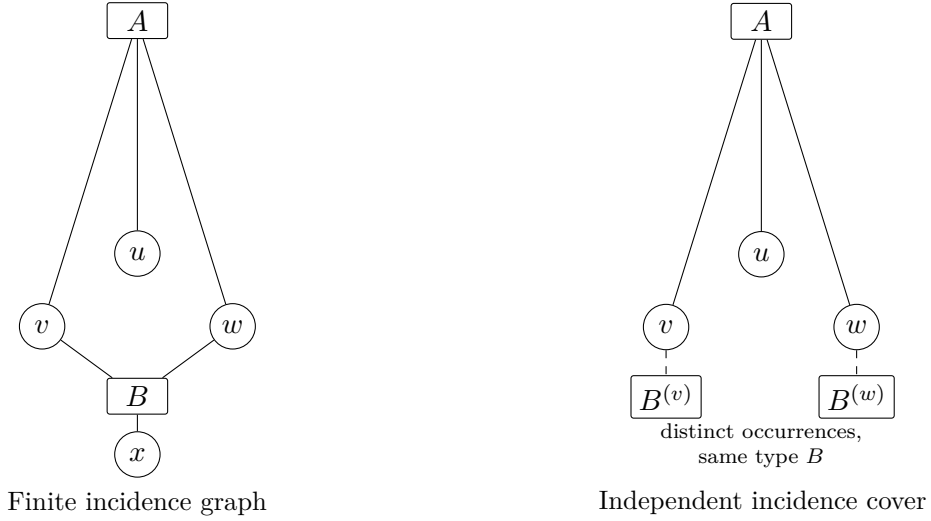

\begin{lemma}[Finite determination]
\label{lem:finite-determination}
For every typed incidence cover, every priority horizon $x\in[0,1]$, and every
fixed finite set of nodes, all decisions concerning those nodes before time
$x$ are determined by a finite random neighbourhood almost surely.
\end{lemma}

\begin{proof}
Each candidate occurrence has at most $B$ candidate grandchildren,
for some finite bound $B$.  A prescribed chain of $\ell+1$ candidate
occurrences can be a dependency chain only if their independent
priorities decrease, with probability $1/(\ell+1)!$.  There are at
most $B^\ell$ such chains starting at any fixed occurrence, and
$B^\ell/(\ell+1)!\to0$.  Thus every fixed finite collection of decisions has a finite
dependency neighbourhood almost surely.
\end{proof}

We truncate at depth $h$, declaring boundary resources free.
Lemma~\ref{lem:finite-determination} gives almost-sure stabilisation of
each fixed decision as $h\to\infty$.  In particular, a test
attached to an exposed but uninvoked occurrence has a well-defined answer.

\subsection{Coupling and cavity survival probabilities}
\label{sec:transfer}

Fix the same root type $v$ in the two queries above.  We shall put
the actual priorities and the cover-occurrence priorities on one
probability space so that $X_v=Y_v$ unless the cover query exposes
two occurrences of the same  candidate type.  The complete
scanner convention is essential: a priority is read at exposure,
even if the corresponding recursive call is never made.

\begin{lemma}[Candidate repetitions precede resource repetitions]
\label{lem:candidate-first}
Within a complete-scanner query, no resource type can be registered
twice before a candidate type has been exposed twice.
\end{lemma}

\begin{proof}
Suppose a second occurrence of resource type $u$ is born from an
invoked candidate occurrence of type $A$.  The complete scanner of
the earlier $u$ has already been exposed: every batch of forward
resources has its scanners exposed before any further candidate is
invoked, and distinct forward resources of one candidate have
distinct types.  Let $P$ be the earlier occurrence's incoming
candidate type, or $\bot$ if it is the root.  If $A\ne P$, the
earlier scanner exposed an occurrence of type $A$; if $A=P$, the
incoming candidate of the earlier occurrence was already exposed.
In the first case this earlier occurrence is distinct from the
one now invoked: a candidate entered through type $u$ cannot
produce another forward resource of type $u$.  In the second case
the incoming candidate is likewise a distinct occurrence.  Thus
type $A$ repeats before the second birth of $u$.
\end{proof}

For this one-root exploration, write $\mathcal C_v$ for the event
that two distinct \emph{exposed} candidate occurrences have the same
 type.  The event refers to exposure, not invocation.

\begin{lemma}[First-label coupling]
\label{lem:direct-coupling}
For each fixed resource $v$, the finite-query answer $X_v$ and the
independent-cover answer $Y_v$ can be coupled so that
\[
 \Prb(X_v\ne Y_v)\le\Prb(\mathcal C_v).
\]
\end{lemma}

\begin{proof}
Sample the actual priorities $(U_A:A\in E(\Gamma))$ independently
and uniformly.  Explore $\mathcal T_v$ in the prescribed scanner
order.  When an occurrence of  candidate type $A$ is exposed
for the first time, give it priority $U_A$; for every later exposed
occurrence of type $A$, use a fresh independent uniform priority.
Complete the unused occurrence priorities independently.

This produces the correct independent-cover law.  Immediately
before each exposure, the type to be exposed is determined by the
previously revealed information and the fixed scanner order.  If its
type is new, $U_A$ is still unread and independent of that history;
otherwise the assigned priority is fresh.  Induction over exposures
therefore gives independent uniform priorities to all cover
occurrences, without conditioning on $\mathcal C_v^c$.

On $\mathcal C_v^c$, every exposed cover priority equals the actual
priority of its type.  Lemma~\ref{lem:candidate-first} also makes
all registered resource types distinct.  Consequently the actual
and cover queries encounter the same complete scanners and use the
same priorities.  Their recursive answers agree by induction over
the almost surely finite dependency neighbourhood of
Lemma~\ref{lem:finite-determination}.  Hence $X_v=Y_v$ on
$\mathcal C_v^c$, proving the bound.
\end{proof}

\subsubsection*{Cavity survival probabilities}

At a root resource, survival before $x$ means that none of its candidate
branches successfully consumes it.  These branches are independent in the
cover, so this event has a product probability.  For a directed incidence
$A\to v$, let $f_{A\to v}(x)$ be the probability that the branch through
$A$ does not consume its parent resource $v$ before time $x$, and let
$q_v(x)$ be the probability that a root resource of type $v$ survives to
time $x$, so $q_v(1)=\Prb(Y_v=1)$.  Define also
\[
 q_{v\to A}(x)
 :=
 \prod_{\substack{C\ni v\\C\ne A}} f_{C\to v}(x).
\]
This is the probability that $v$ remains free before $x$ when its
branch through $A$ is omitted: the \emph{cavity} at the directed
incidence $v\to A$.  Omitting that branch prevents the calculation of
whether $A$ consumes $v$ from using the same candidate again.  We call
$f_{A\to v}$ and $q_{v\to A}$ \emph{messages}: they are precisely
the survival probabilities associated with the two directions of a
resource--candidate incidence.  Distinct branches in the remaining
tree are independent, which is why the product defining $q_{v\to A}$
is exact.

Consequently,
\begin{equation}\label{eq:cavity-root-survival}
 q_v(x)=\prod_{A\ni v}f_{A\to v}(x).
\end{equation}
For $x<1$ we have $f_{A\to v}(x)\ge1-x>0$, and hence also
\begin{equation}\label{eq:cavity-omitted-survival}
 q_{v\to A}(x)=\frac{q_v(x)}{f_{A\to v}(x)}.
\end{equation}
The quotient will only be used where the denominator is known to be
positive.

\begin{proposition}[Exact cavity equations]
\label{prop:cavity-equations}
For every directed incidence $A\to v$,
\begin{equation}\label{eq:cavity-branch-derivative}
 f'_{A\to v}(x)
 =-
 \prod_{u\in A\setminus\{v\}}q_{u\to A}(x).
\end{equation}
Consequently, defining $H_A(x):=\prod_{u\in A}q_{u\to A}(x)$,
we obtain
\begin{equation}\label{eq:cavity-root-derivative}
 q'_v(x)=-\sum_{A\ni v}H_A(x).
\end{equation}
\end{proposition}

\begin{proof}
Let $f_{A\to v}^{[h]}$ denote the branch message with $h$ remaining
candidate generations and free boundary resources.  Write
$q_{u\to A}^{[h-1]}$ for its child cavity messages.  The branch
through $A$ consumes $v$ at priority $t$ if and only if the other
resource children are free immediately before $t$.  Their subtrees are
disjoint, so for $h\ge1$,
\[
 f_{A\to v}^{[h]}(x)
 =1-\int_0^x
 \prod_{u\in A\setminus\{v\}}q_{u\to A}^{[h-1]}(t)\,dt.
\]
The child messages have index $h-1$, not $h$.  Almost-sure finite
determination gives pointwise convergence of these probabilities; bounded
convergence therefore yields
\begin{equation}\label{eq:cavity-integral}
 f_{A\to v}(x)
 =1-\int_0^x
 \prod_{u\in A\setminus\{v\}}q_{u\to A}(t)\,dt.
\end{equation}
The branch messages are continuous by \eqref{eq:cavity-integral}, as are their finite products
$q_{u\to A}$.  We may therefore differentiate \eqref{eq:cavity-integral}, obtaining \eqref{eq:cavity-branch-derivative}.
Differentiating \eqref{eq:cavity-root-survival} for the limiting messages gives
\[
 q'_v
 =
 \sum_{A\ni v}f'_{A\to v}
 \prod_{C\ni v,\,C\ne A}f_{C\to v}
 =-
 \sum_{A\ni v}
 q_{v\to A}
 \prod_{u\in A\setminus\{v\}}q_{u\to A},
\]
which is \eqref{eq:cavity-root-derivative}.
\end{proof}

\subsection{The homogeneous model}

If every resource has degree $d$, symmetry gives a common branch message
$f_d$, with $f_d'=-f_d^{(r-1)(d-1)}$, $f_d(0)=1$, and
$q_d=f_d^d$.  Integrating yields
\[
 q_d(x)=\bigl(1+((r-1)d-r)x\bigr)^{-d/((r-1)d-r)}.
\]
The reference profile is
$Q_D(x)=(1+(r-1)Dx)^{-1/(r-1)}$.  Comparing logarithms gives
$q_d(x)=Q_d(x)(1+O_r(\log d/d))$, uniformly for $0\le x\le1$.
In particular, $q_d(1)\sim((r-1)d)^{-1/(r-1)}$.

\subsection{The moment reduction and the two remaining tasks}

Let $p_{\mathrm{rep}}$ be the probability that a uniform-root cover
query exposes a candidate type twice.  Averaging Lemma~\ref{lem:direct-coupling} over a uniform root, with
$\bar q=N^{-1}\sum_vq_v(1)$, gives
\begin{equation}\label{eq:moment-first}
 \left|\E(Z_\Gamma/N)-\bar q\right|\le p_{\mathrm{rep}}.
\end{equation}

For the second moment, only now choose \emph{two} independent uniform
roots $V_1,V_2$, with replacement.  Their actual queries use the
same candidate priority field $(U_A)$, so their answers are
$X_{V_1},X_{V_2}$; the comparison uses two \emph{independent}
typed covers, with answers $Y_{V_1}^{(1)},Y_{V_2}^{(2)}$.

Couple them by the same first-label rule, exploring the first cover
and then the second: the first exposure of each type $A$ across both
covers uses $U_A$, and every later exposure receives a fresh
independent uniform.  The sequential-independence argument in
Lemma~\ref{lem:direct-coupling} shows that the two covers retain
their joint independent-cover law, even conditional on the sampled
roots.  If no candidate type is repeated within either cover or
across the two, the one-root agreement argument applies to both
answers simultaneously.

Let $p_{\mathrm{shared}}$ be the probability that the two covers expose
a common candidate type.  The probability of a discrepancy between the
two root-answer vectors is at most
$2p_{\mathrm{rep}}+p_{\mathrm{shared}}$.

Conditional on the actual priority field, both independent uniform
roots survive with probability $(Z_\Gamma/N)^2$.  In the independent
covers, the two answers are independent and each has mean $\bar q$.
Comparing the joint event that both answers are true gives
\begin{equation}\label{eq:moment-second}
 \left|\E[(Z_\Gamma/N)^2]-\bar q^2\right|
 \le 2p_{\mathrm{rep}}+p_{\mathrm{shared}}.
\end{equation}
Expanding the square and using \eqref{eq:moment-first}--\eqref{eq:moment-second}
with $Q=Q_D(1)$ gives the estimate needed below:
\begin{equation}\label{eq:moment-L2-bridge}
 \E[(Z_\Gamma/N-Q)^2]
 \le (\bar q-Q)^2+4p_{\mathrm{rep}}+p_{\mathrm{shared}}.
\end{equation}
The independent reference law has
$Q=\Theta_r(D^{-1/(r-1)})$.  Thus the two tasks needed for the
$L^2$ conclusion are precisely
\[
 \sup_v\left|\frac{q_v(1)}Q-1\right|=o(1),
 \qquad p_{\mathrm{rep}}+p_{\mathrm{shared}}=o(Q^2).
\]
Dividing \eqref{eq:moment-L2-bridge} by $Q^2$ gives relative $L^2$ convergence.
Sections~\ref{sec:stability} and \ref{sec:collision} establish the two bounds.

\section{Survival in the independent model} \label{sec:stability}

In this section we show that survival probabilities in the independent model remain
approximately equal to the homogeneous profile $Q_D$ up to priority $x=1$. 
We use the preceding notation: 
survival probability $q_v(x)$, branch message $f_{A\to v}(x)$,
reference profile $Q_D(x)$, resource walk $\mathsf T$ and stationary law $\pi$.
We write $\delta=\max_v|d_v/D-1|$ and $T_D=\log(1+(r-1)D)/(r-1)$.

\begin{theorem}[Independent-cover asymptotics]
\label{thm:cavity-stability}
Fix $r\ge3$, $q>2$, and $C,L,K<\infty$.
Let $\Gamma=\Gamma_n$ be simple $r$-uniform hypergraphs with
$N\le D^K$, $D\to\infty$ and $\delta=o(1)$.  Suppose
\[
 \left\|e^{-(r-1)t\mathsf T}\right\|_{L^q(\pi)\to L^q(\pi)}
 \le C(1+t)^L\qquad(0\le t\le T_D),
\]
\begin{equation}\label{eq:bootstrap-assumption}
 N^{1/q}\left(\delta+\frac{\log D}{D}\right)
 (1+\log D)^{2L+2}=o(1).
\end{equation}
Then, uniformly over $v,A\ni v$ and $0\le x\le1$,
\begin{align}
 q_v(x)&=(1+o(1))Q_D(x),\label{eq:cover-survival-asymptotic}\\
 f_{A\to v}(x)&=1-O_r\!\left(\frac{\log D}{D}\right).
 \label{eq:message-near-one}
\end{align}
\end{theorem}

\subsection{Bootstrap}

We use the following three lemmas, proved in the next subsection.
Write
\begin{equation}\label{eq:deviation-definition}
 q_v(x)=Q_D(x)e^{y_v(x)}
\end{equation}
and set
\[
 \tau:=\frac1{r-1}\log(1+(r-1)Dx).
\]
Since $d\tau/dx=DQ_D(x)^{r-1}$, this maps $[0,1]$ to
$[0,T_D]$, an interval of length $\Theta_r(\log D)$.

For $x<1$ the messages are positive, so \eqref{eq:cavity-omitted-survival} and
Proposition~\ref{prop:cavity-equations} give:

\begin{lemma}[Exact logarithmic equation]
\label{lem:log-equation}
For $x<1$,
\begin{equation}\label{eq:log-derivative}
 \frac{dy_v}{d\tau}
 =1-\frac1D
 \sum_{A\ni v}
 \exp\!\left(\sum_{u\in A\setminus\{v\}}y_u\right)
 \prod_{u\in A}f_{A\to u}^{-1}.
\end{equation}
\end{lemma}

\begin{lemma}[Message estimate]
\label{lem:message-estimate}
There is $h_0=h_0(r)>0$ such that the following holds.  
Suppose that $\max_v|y_v(x)|\le h_0$ on an interval $0\le x\le x_0$.
Then, for $D$ sufficiently large,
\begin{equation}\label{eq:message-correction}
 1-f_{A\to v}(x)
 \le C_r\frac{\log D}{D}
\end{equation}
uniformly over all directed incidences and all $0\le x\le x_0$.  In
particular,
\(
 f_{A\to v}(x)\ge\frac34.
\)
\end{lemma}

\begin{lemma}[Residual estimate]
\label{lem:linearized}
On any interval $0\le\tau\le\tau_0<T_D$ on which
$\sup_{s\le\tau_0}\|y(s)\|_\infty\le h_0$, for every $q\ge1$,
\begin{equation}\label{eq:residual-bound}
 \left\|\frac{dy}{d\tau}+(r-1)\mathsf Ty\right\|_{L^q(\pi)}
 \le C_r\left[
 \delta+\frac{\log D}{D}
 +\|y\|_\infty\|y\|_{L^q(\pi)}
 \right].
\end{equation}
\end{lemma}

\begin{proof}[Proof of Theorem~\ref{thm:cavity-stability}]
Since $y(0)=0$, Duhamel's formula and
Lemma~\ref{lem:linearized} give, on every $[0,t]$ on which
$\|y\|_\infty\le h_0$,
\[
 y(\tau)=\int_0^\tau e^{-(r-1)(\tau-s)\mathsf T}
 \left(y'(s)+(r-1)\mathsf Ty(s)\right)\,ds.
\]
Near regularity gives $\pi_{\min}\ge c/N$, so
$\|z\|_\infty\le C_3N^{1/q}\|z\|_{L^q(\pi)}$.
Write $Y(\tau)=\sup_{0\le s\le\tau}\|y(s)\|_{L^q(\pi)}$.
The hypothesis and \eqref{eq:residual-bound} yield
\begin{equation}\label{eq:duhamel-prefix}
 Y(\tau)\le C_2(1+T_D)^{L+1}
 \left[\delta+\frac{\log D}{D}
       +C_3N^{1/q}Y(\tau)^2\right].
\end{equation}

Fix $C_1>2C_2$.  We claim that
\[
 Y(\tau)\le 2C_1(1+T_D)^{L+1}
             \left(\delta+\frac{\log D}{D}\right)
 \qquad (0\le\tau<T_D).
\]
Initially $y(0)=0$.  On any $[0,t]$ satisfying this bound and
$\|y\|_\infty\le h_0$, substitution into
\eqref{eq:duhamel-prefix} gives
\[
 Y(\tau)\le C_2(1+T_D)^{L+1}
 \left(\delta+\frac{\log D}{D}\right)
 \left[1+4C_3C_1^2N^{1/q}
       \left(\delta+\frac{\log D}{D}\right)
       (1+T_D)^{2L+2}\right].
\]
The bracket is $1+o(1)$ by \eqref{eq:bootstrap-assumption}.
Thus the preceding bound improves to at most
$\frac32C_1(1+T_D)^{L+1}(\delta+\log D/D)$. 
On the same $[0,t]$ we deduce
\[
 \|y\|_\infty\le 2C_1C_3N^{1/q}(1+T_D)^{L+1}
                    \left(\delta+\frac{\log D}{D}\right)
 =o(1)<h_0/2.
\]
The residual estimate has only been used under its hypothesis on $[0,t]$,
If either bootstrap bound failed before $T_D$,
continuity of the finite message products and their positivity for
$x<1$ would permit both strict improvements to persist past the
first failure.  Hence the bounds hold for all $\tau<T_D$.
In particular,
\[
 \sup_{\substack{0\le x<1\\v\in V(\Gamma)}}
 \left|\log\frac{q_v(x)}{Q_D(x)}\right|
 \le C_{r,C,L,q}N^{1/q}
 \left(\delta+\frac{\log D}{D}\right)
 (1+\log D)^{L+1}=o(1).
\]
The message estimate gives $f_{A\to v}(x)\ge3/4$ for $x<1$,
and hence also at $x=1$ by continuity.  Thus
$q_v(1)=\prod_{A\ni v}f_{A\to v}(1)>0$, so the logarithms and
bounds extend to $x=1$.  Now $q_v=Q_D e^{y_v}$ proves
\eqref{eq:cover-survival-asymptotic}, and the message estimate gives
\eqref{eq:message-near-one}.
\end{proof}

\subsection{Proofs of the auxiliary lemmas}

\begin{proof}[Proof of Lemma~\ref{lem:log-equation}]
Since $Q_D'/Q_D=-DQ_D^{r-1}$, we have
$dy_v/dx=q'_v/q_v+DQ_D^{r-1}$.
For $A\ni v$, \eqref{eq:cavity-omitted-survival} gives
\[
 \frac{H_A}{q_v}
 =Q_D^{r-1}
 \exp\!\left(\sum_{u\in A\setminus\{v\}}y_u\right)
 \prod_{u\in A}f_{A\to u}^{-1}.
\]
Substituting \eqref{eq:cavity-root-derivative} and dividing by $d\tau/dx=DQ_D^{r-1}$
gives \eqref{eq:log-derivative}.
\end{proof}

\begin{proof}[Proof of Lemma~\ref{lem:message-estimate}]
Until a message first reaches $3/4$, equations \eqref{eq:cavity-branch-derivative}, \eqref{eq:cavity-omitted-survival}, and
\eqref{eq:deviation-definition} give
\[
 |f'_{A\to v}(x)|
 \le
 \left(\frac43e^{h_0}Q_D(x)\right)^{r-1}
 \le C_rQ_D(x)^{r-1}.
\]
Since $Q_D(t)^{r-1}=1/(1+(r-1)Dt)$, integration gives
\[
 1-f_{A\to v}(x)
 \le C_r\int_0^xQ_D(t)^{r-1}\,dt
 \le C_r\frac{\log(1+(r-1)D)}{(r-1)D}
 =O_r\!\left(\frac{\log D}{D}\right).
\]
For sufficiently large $D$ this is below $1/4$, preventing a message
from reaching $3/4$ and proving \eqref{eq:message-correction}.
\end{proof}

\begin{proof}[Proof of Lemma~\ref{lem:linearized}]
For $A\ni v$, write
\[
 z_{A,v}:=\sum_{u\in A\setminus\{v\}}y_u,
 \qquad
 m_A:=\prod_{u\in A}f_{A\to u}^{-1}.
\]
By the message estimate, $m_A=1+O_r(\log D/D)$ uniformly.  For $\|y\|_\infty\le h_0$, Taylor's theorem gives
\[
 |e^{z_{A,v}}-1-z_{A,v}|
 \le
 C_r\|y\|_\infty
 \sum_{u\in A\setminus\{v\}}|y_u|.
\]
The cross term $(m_A-1)z_{A,v}$ has absolute value at most
$C_r(\log D/D)\sum_{u\in A\setminus\{v\}}|y_u|$.
Since $\|y\|_\infty\le h_0<1$ and $d_v=O_r(D)$, summing over
$A\ni v$ and dividing by $D$ contributes $O_r(\log D/D)$, which
is included in the additive error rather than the quadratic remainder.

Substituting in \eqref{eq:log-derivative},
\[
 \frac{dy_v}{d\tau}
 =1-\frac{d_v}{D}
 -\frac1D\sum_{A\ni v}
 \sum_{u\in A\setminus\{v\}}y_u
 +E_v,
\]
where
\[
 |E_v|
 \le
 C_r\left[
 \frac{d_v}{D}\frac{\log D}{D}
 +\frac{\|y\|_\infty}{D}
 \sum_{A\ni v}\sum_{u\in A\setminus\{v\}}|y_u|
 \right].
\]
The middle linear term is $-(r-1)(d_v/D)(\mathsf T y)(v)$.
Consequently, the residual in \eqref{eq:residual-bound} satisfies
\[
 (y'+(r-1)\mathsf Ty)_v
 =\left(1-\frac{d_v}{D}\right)
  \left(1+(r-1)(\mathsf T y)(v)\right)+E_v.
\]
Since $\|y\|_{L^q(\pi)}\le\|y\|_\infty\le1$ and the stationary
Markov operator $\mathsf T$ is a contraction on $L^q(\pi)$, taking
norms of this identity and the bound on $E_v$ proves \eqref{eq:residual-bound}.
\end{proof}

\section{Bounding collisions}
\label{sec:collision}

We return to the candidate system $\Gamma$ of Section~\ref{sec:cover}.
As before, we write $N=|V(\Gamma)|$ and $M=|E(\Gamma)|$.
We write $d_*:=\max_vd_v$ for the maximum resource degree.
We assume $D/2\le d_v\le3D/2$
for every resource $v$, with $D\ge D_0(r)$, and put
$B:=(r-1)(d_*-1)$ and $R_0:=1+(r-1)D$.  In particular $B=\Theta_r(D)$ and $M=\Theta_r(ND)$.  We use the cover survival bounds
\begin{equation}\label{eq:cover-survival-bounds}
 q_v(t)\le\tfrac54Q_D(t),\qquad
        f_{A\to v}(t)\ge\tfrac34
\end{equation}
for every type and $t\in[0,1]$.  These are consequences of
Theorem~\ref{thm:cavity-stability} in our eventual application.

At the terminal time the survival probability of a resource is of order
$D^{-1/(r-1)}$.  To control the second moment, we must compare the
actual and independent models with an absolute error smaller than its
square, $D^{-2/(r-1)}$; an error merely tending to zero would not suffice.
The following theorem makes this comparison without assuming that two
candidates intersect in at most one resource.  Recall that $\chi_R(\Gamma)$
is the proportion of candidates on incidence cycles of length at most $R$.

\begin{theorem}[Collision comparison]
\label{thm:transfer}
For each fixed $r\ge3$ there are constants $C_r<\infty$ and
positive integers $\alpha_r,\gamma_r$ such that for every
$\kappa\ge2$, with $R=4\kappa+8$ and
$\bar q=N^{-1}\sum_vq_v(1)$, the number $Z$ of uncovered resources
satisfies
\begin{align}
 \left|\E(Z/N)-\bar q\right|
 &\le C_rR_0^{\alpha_r}(\chi_R+B^{-\kappa}),\label{eq:collision-moment-first}\\
 \left|\E[(Z/N)^2]-\bar q^2\right|
 &\le 2C_rR_0^{\alpha_r}(\chi_R+B^{-\kappa})
     +\frac{C_rR_0^{\gamma_r}}M.\label{eq:collision-moment-second}
\end{align}
\end{theorem}

\begin{proof}[Proof of Theorem~\ref{thm:fixed-host}, assuming Theorem~\ref{thm:transfer}]
Theorem~\ref{thm:cavity-stability} gives
\[
 q_v(1)=(1+o(1))Q_D(1)
 =(1+o(1))((r-1)D)^{-1/(r-1)}.
\]
Thus $\bar q=(1+o(1))Q_D(1)$, and the two bounds \eqref{eq:cover-survival-bounds} hold for
sufficiently large $n$.

Let $\alpha_r,\gamma_r$ be supplied by Theorem~\ref{thm:transfer}.
Choose $\kappa_r>\alpha_r+2/(r-1)+2$ and set
$R_r:=4\kappa_r+8$.  Enlarge the exponents in the statement of
Theorem~\ref{thm:fixed-host}, if necessary, to satisfy
\[
 A_r>\alpha_r+1,\qquad B_r>\gamma_r+1.
\]
Since $B=\Theta_r(D)$, the anchored-counting contribution obeys
$R_0^{\alpha_r}B^{-\kappa_r}=o(D^{-2/(r-1)})$.
The short-cycle hypothesis gives
$\chi_{R_r}R_0^{\alpha_r}=o(D^{-2/(r-1)})$,
while $D^{B_r}/M=o(D^{-2/(r-1)})$ controls the cross-root term.
Thus Theorem~\ref{thm:transfer} gives
$|\E(Z/N)-\bar q|=o(D^{-2/(r-1)})$ and
$|\E[(Z/N)^2]-\bar q^2|=o(D^{-2/(r-1)})$.
As $\bar q=(1+o(1))Q_D(1)$ and
$Q_D(1)^2=\Theta_r(D^{-2/(r-1)})$, the moment bound
\eqref{eq:moment-L2-bridge} gives
$\E[(Z/(NQ_D(1))-1)^2]=o(1)$.
The expectation, variance and convergence-in-probability assertions
follow.
\end{proof}

\subsection{Proof of the collision comparison}
\label{sec:collision-comparison-proof}

The moment comparison in Section~\ref{sec:cover} reduces
Theorem~\ref{thm:transfer} to two probabilities: for a uniform-root
cover query, repetition of an exposed candidate type, and for two
independent cover queries, exposure of a common candidate type.
With $R=4\kappa+8$, it suffices to prove
\begin{align}
 p_{\mathrm{rep}}&\le C_rR_0^{\alpha_r}(\chi_R+B^{-\kappa}),\label{eq:first-collision-bound}\\
 p_{\mathrm{shared}}&\le C_rR_0^{\gamma_r}/M.\label{eq:cross-root-collision}
\end{align}

Fix a root resource type $v$ in the independent cover.  For distinct
candidate occurrences $x,y$, let $W=P(v,x)\cup P(v,y)$ be the union
of their incidence-tree paths from the root.  A candidate occurrence
through which one of these paths continues to a forward resource must
be invoked; either $x$ or $y$ may instead be merely exposed in a
scanner.  Suppress the intermediate resource occurrences when counting
steps between successive required candidate invocations.  These form a
common chain of length $h$ and two continuations of lengths
$\ell_1,\ell_2$, with $L=h+\ell_1+\ell_2$; any of the three
lengths may be zero.

A \emph{shape} records this rooted two-path topology, the positions at
which $x,y$ are exposed (including when one is an ancestor of the other),
and, at each step, the choice of incoming and outgoing incidence
positions in an $r$-candidate.  It does not record any  resource
or candidate types.  For fixed $h,\ell_1,\ell_2$, the branching
location and the exposure locations have at most $(L+3)^3$ choices,
and each candidate step has at most $O_r(1)$ incidence-position
choices.  Thus the number of shapes is at most $K_r^{L+1}$ for a
sufficiently large constant $K_r$ depending only on $r$.

A \emph{first-collision witness} has $x$ and $y$ of the same 
candidate type, with all other candidate types distinct from each
other and from that type, and all resource types distinct.  The first
repeated candidate type in a cover query gives such a witness by
Lemma~\ref{lem:candidate-first}.

\begin{figure}[t]
\centering
\begin{tikzpicture}[>=latex, font=\small, thick,
  every node/.style={inner sep=2pt}]
\node (v) at (0,0) {$v$};
\node (z) at (2.2,0) {$z$};
\node (up) at (3.65,0.95) {$\cdots$};
\node (y) at (5.1,0.95) {$y$};
\node (a0) at (3.5,-1.05) {$A_0$};
\node (yy) at (5.6,-1.05) {$Y$};
\node (x) at (7.15,-1.05) {$x$};
\draw[->] (v)--node[above]{common chain}(z);
\draw[->] (z)--(up);
\draw[->] (up)--(y);
\draw[->] (z)--(a0);
\draw[->,dashed] (a0)--node[above]{$\kappa$ steps}(yy);
\draw[->,dotted] (yy)--node[above]{scanner}(x);
\draw[densely dotted, rounded corners] (y.north west)++(-.18,.19) rectangle ++(.65,-.64);
\draw[densely dotted, rounded corners] (x.north west)++(-.17,.19) rectangle ++(.65,-.64);
\end{tikzpicture}
\caption{A two-path witness in the non-ancestor case.  The selected
occurrences $x,y$ have the same  candidate type; the dashed block
is reconstructed from its two boundary types in Lemma~\ref{lem:private-block}.
The paths depict successive candidate invocations, while the final
dotted step depicts scanner exposure.}
\label{fig:collision-witness}
\end{figure}
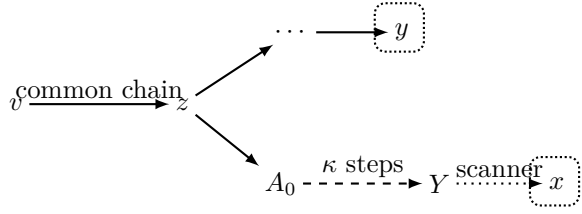

Put $c_r=\frac{20}{9}(\frac53)^{r-2}$ and
$\mu_D=c_r\log R_0/((r-1)D)$, so that
$B\mu_D\le\frac32c_r\log R_0$.  Choose $K_r$ sufficiently
large for both the shape count and the incidence choices in the
anchored count below.  We use three estimates, proved below.

\begin{lemma}[Two-path visitation]
\label{lem:two-path-visit}
For a fixed $W=P(v,x)\cup P(v,y)$ as above, the probability that
$x$ and $y$ are both exposed is at most
$C_rR_0^3\mu_D^{h+\ell_1+\ell_2}/(h!\,\ell_1!\,\ell_2!)$.
\end{lemma}

For each  candidate type $A$, let $p_A$ be the probability that
a cover query with a uniformly chosen root exposes an occurrence of type $A$.
\begin{lemma}[Uniform exposure of one candidate type]
\label{lem:one-label-exposure}
There are constants $C_r,H_r$ depending only on $r$ such that
$p_A\le C_rR_0^{H_r}/N$ for every candidate type $A$.
\end{lemma}

\begin{lemma}[Anchored count for long first collisions]
\label{lem:private-block}
Let $\kappa\ge2$, $R=4\kappa+8$, and fix a two-path shape with
$L$ successive-invocation steps and distance greater than $R$
between $x$ and $y$ in the incidence tree.  There are at most $C_rNR_0^4(K_rB)^LB^{-\kappa}$ first-collision
witnesses of this shape for which no candidate type on either selected
path lies on a simple incidence cycle of length at most $R$.
\end{lemma}

The visitation and exposure lemmas are proved in Section~\ref{sec:visitation-exposure},
and the anchored count in Section~\ref{sec:anchored-counting}.
\begin{proof}[Proof of Theorem~\ref{thm:transfer}, assuming Lemmas~\ref{lem:two-path-visit}--\ref{lem:private-block}]
By \eqref{eq:moment-first} and \eqref{eq:moment-second}, it suffices to prove \eqref{eq:first-collision-bound} and \eqref{eq:cross-root-collision}.
Take the first repeated candidate type and its earlier occurrence.
The paths to the two selected occurrences form a first-collision witness.
If their connecting path has at most $R$ incidence edges, its 
image closes to a simple incidence cycle of length at most $R$:
all interior candidate and resource types are distinct, while
the two selected candidates have the same type.  Let $\mathcal C_R$ be
the set of candidate types on such cycles, and let $\mathcal B_R$
be the event that a uniform-root cover query exposes a type in
$\mathcal C_R$.  Since $|\mathcal C_R|=M\chi_R$,
Lemma~\ref{lem:one-label-exposure} gives
\[
 \Prb(\mathcal B_R)
 \le\sum_{A\in\mathcal C_R}p_A
 \le\frac{C_r|\mathcal C_R|R_0^{H_r}}N
 \le C_r\chi_RR_0^{H_r+1}.
\]

Let $\mathcal W_R$ be the set of long first-collision witnesses,
rooted at arbitrary resource types, whose selected paths contain no
type in $\mathcal C_R$.  For each such witness $W$ rooted at $v$,
write $\mathcal V_W$ for the event that the cover query rooted at $v$
exposes its two selected candidate occurrences.  If a first repetition
is not accounted for by $\mathcal B_R$, it gives a witness in
$\mathcal W_R$.  Thus
$p_{\mathrm{rep}}\le\Prb(\mathcal B_R)+N^{-1}
\sum_{W\in\mathcal W_R}\Prb(\mathcal V_W)$.
For each unlabelled shape, multiply the count in
Lemma~\ref{lem:private-block} by the visitation bound in
Lemma~\ref{lem:two-path-visit}, average over roots, and sum shapes
and lengths.  There are at most $K_r^{L+1}$ shapes for
$h,\ell_1,\ell_2$, so
\[
 \begin{aligned}
 \frac1N\sum_{W\in\mathcal W_R}\Prb(\mathcal V_W)
 &\le C_rR_0^7B^{-\kappa}\sum_{h,\ell_1,\ell_2\ge0}
   \frac{(K_r^2B\mu_D)^{h+\ell_1+\ell_2}}{h!\ell_1!\ell_2!}\\
 &=C_rR_0^7B^{-\kappa}\exp(3K_r^2B\mu_D).
 \end{aligned}
\]
Since $B\mu_D\le\frac32c_r\log R_0$, this is at most
$C_rR_0^{7+(9/2)c_rK_r^2}B^{-\kappa}$.
Choose any integer
$\alpha_r>\max\{H_r+1,\,7+(9/2)c_rK_r^2\}$.
Combining the two cases establishes \eqref{eq:first-collision-bound}.

Independence of the two roots and their occurrence priorities gives
\[
 p_{\mathrm{shared}}\le\sum_{A\in E(\Gamma)}p_A^2
 \le\frac{C_rMR_0^{2H_r}}{N^2}.
\]
Since $rM=\sum_vd_v$ and $d_v\le3D/2$,
$M/N\le C_rD$, whence
$M/N^2\le C_rR_0^2/M$.
Taking any integer $\gamma_r>2H_r+2$ establishes \eqref{eq:cross-root-collision}.

Substituting \eqref{eq:first-collision-bound} and \eqref{eq:cross-root-collision} into \eqref{eq:moment-first} and \eqref{eq:moment-second} proves
\eqref{eq:collision-moment-first}--\eqref{eq:collision-moment-second}.
\end{proof}

\subsection{Visitation and exposure}\label{sec:visitation-exposure}

\begin{proof}[Proof of Lemma~\ref{lem:two-path-visit}]
Fix $W$ before sampling priorities.  Let $I(W)$ be the required
invoked candidate occurrences, and let $\mathcal V_W$ be the event
that the query exposes $x,y$.  Write $E(W)$ for the steps $A\to C$
between successive required invocations; thus
$|E(W)|=h+\ell_1+\ell_2$.  At most one candidate invocation below
the root has two required successors.  Let $E_{\rm br}$ consist of
its outgoing steps, if any, so $|E_{\rm br}|\le2$.

Fix the vector $U_I=(U_C)_{C\in I(W)}$ at
$t=(t_C)_{C\in I(W)}\in[0,1]^{I(W)}$.  Write $\Prb_{\rm off}$
for probability over the remaining independent occurrence priorities,
and $\mathcal V_W(t)$ for the exposure event with $U_I=t$.
Put $\mathcal D_W=\{t:t_C<t_A\text{ for all }(A,C)\in E(W)\}$.
The event $\mathcal V_W(t)$ is impossible unless $t\in\mathcal D_W$.

Set $w(s)=c_rQ_D(s)^{r-1}=c_r/(1+(r-1)Ds)$, so that
$\mu_D=\int_0^1w(s)\,ds$.  Consider a step $A\to C$ with
$C$ the only required successor of $A$.  Let $p$ be the incoming
resource of $A$, and $u$ the forward resource whose scanner contains
$C$.  For $C$ to be invoked at priority $t_C$, all other entries of
priority below $t_C$ in the merged forward scanner of $A$ must
return false.  Given $U_I=t$, these answers depend only on the
candidate branches other than $A$ and $C$ at $u$, and other than $A$
at each $z\in A\setminus\{p,u\}$.  By the cavity product
\eqref{eq:cavity-root-survival}, their joint probability is
\[
 \begin{aligned}
 &\frac{q_u(t_C)}{f_{A\to u}(t_C)f_{C\to u}(t_C)}
   \prod_{z\in A\setminus\{p,u\}}
      \frac{q_z(t_C)}{f_{A\to z}(t_C)}\\
 &\qquad\le \frac{20}{9}Q_D(t_C)
       \left(\frac53Q_D(t_C)\right)^{r-2}
       =w(t_C).
 \end{aligned}
\]
Indeed, those branches are distinct in the incidence cover, even
if some have the same  candidate type.  For different ordinary
steps, they lie in disjoint subtrees off the selected paths, so their
scanner outcomes are conditionally independent.

We discard all root-scanner conditions and the conditions at the
possible branching invocation.  By the preceding calculation,
\[
 \begin{aligned}
 \Prb_{\rm off}(\mathcal V_W(t))
 &\le\mathbf1_{\mathcal D_W}(t)
      \prod_{(A,C)\in E(W)\setminus E_{\rm br}}w(t_C)\\
 &\le \left(\frac{R_0}{c_r}\right)^2
      \mathbf1_{\mathcal D_W}(t)
      \prod_{(A,C)\in E(W)}w(t_C),
 \end{aligned}
\]
since $w(s)\ge c_r/R_0$ and $|E_{\rm br}|\le2$.

Integrate over the independent priorities in $I(W)$.  Dropping
ordering constraints between the common chain and its two
continuations allows three decreasing-priority integrals; the
at most two root-invoked candidates contribute unweighted integrals
of at most one.  Hence
\[
 \begin{aligned}
 \Prb(\mathcal V_W)
 &=\int_{[0,1]^{I(W)}}\Prb_{\rm off}(\mathcal V_W(t))
                     \prod_{C\in I(W)}dt_C\\
 &\le \left(\frac{R_0}{c_r}\right)^2
      \int_{\mathcal D_W}\prod_{(A,C)\in E(W)}w(t_C)
                            \prod_{C\in I(W)}dt_C\\
 &\le \left(\frac{R_0}{c_r}\right)^2
      \frac{\mu_D^{h+\ell_1+\ell_2}}{h!\,\ell_1!\,\ell_2!}
 \le C_rR_0^3\frac{\mu_D^{h+\ell_1+\ell_2}}{h!\,\ell_1!\,\ell_2!}.
 \end{aligned}\qedhere
\]
\end{proof}

\begin{proof}[Proof of Lemma~\ref{lem:one-label-exposure}]
The root scanner exposes $A$ only if the uniform root is one of its
$r$ resources.  Otherwise a resource $u\in A$ is registered by some
invoked candidate $Y\ni u$, with at most $rd_*$ choices for $(u,Y)$.
For a prescribed path $P$ of $j$ successive invocations leading to
$Y$, let $\mathcal V_P$ be the event that the path is followed.  The
ordinary-step calculation in the proof of Lemma~\ref{lem:two-path-visit}
(applied to this single path) and integration over its decreasing
priorities give
\[
 \Prb(\mathcal V_P)
 \le\int_{1>t_1>\cdots>t_j>0}
          \prod_{i=1}^jw(t_i)\,dt_1\cdots dt_j
 =\frac{\mu_D^j}{j!}.
\]
The first candidate has an unweighted priority integral of at most
one.  There are at most $C_r(K_rB)^j$ possible paths, including the
bounded choice of root incidence.  Since $d_*=O_r(R_0)$ and
$B\mu_D\le\frac32c_r\log R_0$, take any fixed integer
$H_r>5+3c_rK_r/2$.  Averaging over the $N$ root types and summing
gives
\[
 p_A\le\frac rN+\frac{C_rrd_*}{N}e^{K_rB\mu_D}
 \le\frac{C_rR_0^{H_r}}N.\qedhere
\]
\end{proof}

\subsection{Anchored witness counting}\label{sec:anchored-counting}

\begin{proof}[Proof of Lemma~\ref{lem:private-block}]
Let $z$ be the last common node on the paths to $x,y$, and choose
the longer arm with a fixed tie-break.  Since the distance from $x$
to $y$ exceeds $4\kappa+8$, this arm contains more than
$2\kappa+4$ incidence edges.  Its terminal occurrence, say $x$, is
exposed in the scanner of a resource $u$ born from an invoked candidate
$Y$.  The last $\kappa$ successive-invocation steps before $Y$ form
a block
\[
 A_0,u_1,A_1,\ldots,u_\kappa,A_\kappa=Y
\]
strictly after $z$, with $x$ exposed through $Y-u-x$.

Retain the root, the incidence choices outside the open interior of
this block, and the candidate type $A$ of the other selected occurrence
$y$.  Since $x$ has the same type $A$, specifying its terminal scanner
requires choosing $u\in A$ and $Y\ni u$, in at most $rd_*$ ways.
Thus the retained data determine both boundary candidate types
$A_0,Y$.  There are $N$ choices for the root, at most two initial
candidates and two terminal scanner entries giving at most four
degree-sized choices, and at most $K_rB$ incidence choices for each
of the remaining $L-\kappa$ invocation steps.  Hence there are at
most $C_rNR_0^4(K_rB)^{L-\kappa}$ encodings.

The missing interior is unique.  Otherwise the same boundary
candidates would be connected by two distinct simple  incidence
paths of length $2\kappa$.  Their union contains an incidence cycle
of length at most $4\kappa\le R$ containing a candidate type on the
selected path, contrary to the hypothesis.  Once the  path is
fixed, its lift into the typed cover is unique from the occurrence of
$A_0$ and its incoming incidence.  Hence the exterior encoding is
injective, and the count is at most
$C_rNR_0^4(K_rB)^LB^{-\kappa}$.
\end{proof}

\part{Application to hypergraph removal}

In this part we prove the removal result for a fixed $k$-graph $F$.
In Section~\ref{sec:reduction} we state our result and deduce it from the general
random greedy matching theorem, modulo hypothese verified in the remaining sections:
the stability hypothesis is verified in Section~\ref{sec:persistent-chains}
and the short cycle hypothesis in Section~\ref{sec:cycles}.

\section{The removal theorem and handoff} \label{sec:reduction}

We now apply the random greedy matching theorem to the removal process 
for a fixed strictly $k$-balanced $k$-graph $F$ from $K_n^{(k)}$.  At each step a uniformly random
surviving copy of $F$ is selected and all its edges are deleted.
Write $H_{\mathrm{fin}}$ for the terminal $k$-graph.
Let
\[
 m:=v(F),\qquad r:=e(F),\qquad
 \rho_F:=\frac{r-1}{m-k}.
\]
The hypothesis that $F$ is \emph{strictly $k$-balanced} means that
$r\ge3$ and every proper subhypergraph $F'\subsetneq F$ with at least
two edges satisfies
\[
 \frac{e(F')-1}{v(F')-k}<\rho_F.
\]

\begin{theorem}[Terminal $L^2$ asymptotic]
\label{thm:main}
Let $F$ be a fixed strictly $k$-balanced $k$-uniform hypergraph, with
$m=v(F)$, $r=e(F)$, and $\rho_F=(r-1)/(m-k)$.  Define
\[
 C_F=\frac1{k!}
 \left(\frac{r(r-1)k!}{\operatorname{aut}(F)}\right)^{-1/(r-1)}.
\]
Then
\[
 \E\left[\left(\frac{e(H_{\mathrm{fin}})}{n^{k-1/\rho_F}}-C_F\right)^2\right]
 \longrightarrow0.
\]
In particular, $e(H_{\mathrm{fin}})=(C_F+o_{\Prb}(1))n^{k-1/\rho_F}$.
\end{theorem}

\subsection{The standard matching representation}\label{subsec:removal-matching}

Let $H(i)$ be the random $F$-removal process starting from
$H(0)=K_n^{(k)}$, kept constant after termination, and let
$\mathcal F_i$ be its history through step $i$.  Fix $\eta>0$ sufficiently
small, to be chosen in the proof, and put
\[
 p_\star=n^{-1/\rho_F+\eta},\qquad
 i^\star=\left\lfloor\frac{\binom nk-p_\star n^k/k!}{r}\right\rfloor,
 \qquad H_\star=H(i^\star).
\]
On the event that the process reaches $i^\star$, form an $r$-uniform
hypergraph $\Gamma_\star$ whose vertices are the edges of $H_\star$
and whose candidates are the sets of $r$ edges belonging to copies
of $F$ in $H_\star$.  Strict balance excludes isolated vertices of
$F$, so a copy is determined by its edge set and $\Gamma_\star$ is simple.
Write $N,M$ for the numbers of resources and candidates, and $d_e$
for the degree of a resource $e$.  On the event of reaching the handoff,
\begin{equation}\label{eq:resource-count}
 N=\binom nk-ri^\star=\frac{p_\star n^k}{k!}+O(1).
\end{equation}

Conditional on $\mathcal F_{i^\star}$ and on reaching the handoff,
the continuation is random greedy matching on $\Gamma_\star$:
we assign independent continuous priorities to the remaining candidate copies
of $F$ and accept them in increasing order whenever their edges have
not been deleted. The resources are the edges of $H_{\mathrm{fin}}$.

Put $a=m-k$, $b=r-1$, and $\rho=\rho_F=b/a$, and define
\begin{equation}\label{eq:candidate-reference-degree}
 D_\star:=\frac{rk!}{\operatorname{aut}(F)}n^ap_\star^b.
\end{equation}

We use the following notation for counts at the handoff.
A \emph{rooted template} $(A,I)$ is a $k$-graph $A$ with a
specified vertex set $I\subseteq V(A)$.  Given an injection
$\psi:I\hookrightarrow[n]$, let $\Phi_{A,\psi}$ count injective
extensions into $H_\star$ that realise every edge of $A$
not contained in $I$; nonedges are not tested.  Write
\[
 \widehat p(i)=\frac{k!(\binom nk-ri)}{n^k},\qquad
 \widehat\phi_{A,I}(i)
 =n^{|V(A)|-|I|}\widehat p(i)^{e(A)-e(A[I])}.
\]
A rooted subtemplate $(B,I)\subseteq(A,I)$ contains $I$
and may omit other vertices or edges.  For
$I\subsetneq U\subseteq V(A)$ set
\[
 d_{A,I}(U)=\frac{e(A[U])-e(A[I])}{|U|-|I|}.
\]
We use the balance conditions
\[
\begin{aligned}
 (A,I)\text{ balanced}
 &\iff d_{A,I}(U)\le d_{A,I}(V(A))
 &&(I\subsetneq U\subseteq V(A)),\\
 (A,I)\text{ strictly balanced}
 &\iff d_{A,I}(U)<d_{A,I}(V(A))
 &&(I\subsetneq U\subsetneq V(A)).
\end{aligned}
\]
These inequalities also control subtemplates with omitted tested
edges, since omitting edges decreases their rooted densities.

The Joos--K\"uhn counting estimates, recorded in
Section~\ref{sec:handoff-inputs}, hold on an event $\mathcal G_\star$
of probability at least $1-\exp\{-(\log n)^{4/3}\}$.  On this event
the process reaches $i^\star$ and the candidate degrees are nearly
constant, as follows.

\begin{corollary}[Near-regular candidate degrees]
\label{cor:near-regular-handoff}
On $\mathcal G_\star$,
\begin{equation}\label{eq:resource-degrees}
 d_e=(1+n^{-\Omega_{F,\eta}(1)})D_\star
\end{equation}
uniformly in $e\in E(H_\star)$.  Consequently
$M=(1+o(1))ND_\star/r$ and $D_\star=n^{b\eta+o(1)}$.
\end{corollary}

\begin{proof}
For each role edge $f\in E(F)$ and bijection $\sigma:f\to e$,
let $\Phi_{F,\sigma}$ count extensions to a copy of $F$ whose other
edges survive in $H_\star$.  Every candidate through $e$ occurs
$\operatorname{aut}(F)$ times in the sum
\[
 d_e=\frac1{\operatorname{aut}(F)}
 \sum_{f\in E(F)}\sum_{\sigma:f\overset\sim\longrightarrow e}
 \Phi_{F,\sigma}.
\]
Proposition~\ref{prop:JK-handoff}(i) estimates each of these $rk!$
terms by $(1+n^{-\Omega_{F,\eta}(1)})n^ap_\star^b$.
The conclusions follow from $rM=\sum_ed_e$.
\end{proof}

\subsection{Proof of the theorem}

\begin{proof}[Proof of Theorem~\ref{thm:main}, assuming Theorems~\ref{thm:resource-semigroup} and~\ref{thm:cyclic-sparsity}]
We apply Theorem~\ref{thm:fixed-host} conditionally on the handoff.
Its hypotheses are near-regular resource degrees, stability of the
resource walk and short-cycle sparsity, with the stated quantitative
smallness conditions.  The degree estimate is
Corollary~\ref{cor:near-regular-handoff}; the other two are
Theorems~\ref{thm:resource-semigroup} and~\ref{thm:cyclic-sparsity},
proved below.  Assuming them, the first calculation determines the constant: by
\eqref{eq:resource-count} and \eqref{eq:candidate-reference-degree},
\[
 NQ_{D_\star}(1)
 \sim\frac{n^kp_\star}{k!}
 \left(\frac{r(r-1)k!}{\operatorname{aut}(F)}
       n^{m-k}p_\star^{r-1}\right)^{-1/(r-1)}
 =C_F n^{k-1/\rho_F}.
\]
In particular, the handoff density cancels.
Fix the constants $A_r,B_r,R_r$ in Theorem~\ref{thm:fixed-host},
and choose $\eta>0$ sufficiently small for the following estimates.
Using the handoff counts recorded in Section~\ref{sec:handoff-inputs},
Theorem~\ref{thm:cyclic-sparsity} gives
$\chi_{R_r}(\Gamma_\star)\le n^{-1/b+C_{F,R_r}\eta}$,
while Corollary~\ref{cor:near-regular-handoff} gives
$D_\star=n^{b\eta+o(1)}$ and
$M=n^{k-1/\rho_F+r\eta+o(1)}$.
Since $k-1/\rho_F=(kr-m)/(r-1)>0$ (as $F$ is not a matching),
we may take $\eta$ small enough that
$-1/b+(C_{F,R_r}+bA_r)\eta<-2\eta$ and
$-(k-1/\rho_F)+(bB_r-r)\eta<-2\eta$.
Consequently
\[
 D_\star^{A_r}\chi_{R_r}(\Gamma_\star)
 +\frac{D_\star^{B_r}}M
 =o(D_\star^{-2/b}).
\]

On $\mathcal G_\star$ the corollary also gives
$\delta_n:=\max_e|d_e/D_\star-1|\le n^{-c_0}$
for some $c_0=c_0(F,\eta)>0$, and
$N=n^{k-1/\rho_F+\eta+o(1)}\le D_\star^K$ for a fixed $K$.
Put $s=\max\{|f\cap g|:f,g\in E(F),\ f\ne g\}$.
Theorem~\ref{thm:resource-semigroup} gives, for every fixed $q>2$,
\[
 \|e^{-(r-1)t\mathsf T}\|_{L^q(\pi)\to L^q(\pi)}
 \le C_{F,\eta,q}(1+t)^{s+1}
 \qquad(0\le t\le T_{D_\star}).
\]
Choose $q$ sufficiently large that
$(k-1/\rho_F+\eta+1)/q<\tfrac12\min\{c_0,b\eta\}$.
Then
\[
 N^{1/q}\left(\delta_n+\frac{\log D_\star}{D_\star}\right)
 (1+\log D_\star)^{2s+4}=o(1).
\]
These are the three hypotheses of Theorem~\ref{thm:fixed-host}.

Put $N_\star=\binom nk-ri^\star$, so $N=N_\star$ on
$\mathcal G_\star$.  The matching theorem and the conditional
random-priority representation above give a deterministic
$\epsilon_n\to0$ such that
\[
 \E\left[\left.
  \left(\frac{e(H_{\mathrm{fin}})}{N_\star Q_{D_\star}(1)}-1\right)^2
  \right|\mathcal F_{i^\star}\right]\le\epsilon_n
 \quad\text{on }\mathcal G_\star.
\]
Since $N_\star Q_{D_\star}(1)/n^{k-1/\rho_F}\to C_F$ and 
$\Prb(\mathcal G_\star^c) \le\exp\{-(\log n)^{4/3}\}$
we conclude that
\[
 \E\left[\left(
 \frac{e(H_{\mathrm{fin}})}{n^{k-1/\rho_F}}-C_F
 \right)^2\right] =o(1).
 \qedhere
\]
\end{proof}

\section{Stability of the resource walk}
\label{sec:persistent-chains}

This section verifies the stability hypothesis 
for the resource walk that moves from one surviving edge 
through a copy of $F$ to another edge of that copy.  
We first prove the stability estimate
assuming two mixing statements and an elementary operator bound.
We then deduce the mixing statements from counts of chains of copies of $F$.
Every error $n^{-c}$ below has a positive exponent depending on the
fixed parameters; its value may decrease from line to line.

\subsection{Proof modulo lemmas}

Recall that $r=e(F)$, that $d_e$ is the number of surviving
copies of $F$ containing the resource $e$, and that $M$ is the
number of such copies.  The resource walk $\mathsf T$ is reversible
for $\pi_k(e)=d_e/(rM)$.  As $F$ is strictly $k$-balanced we have
\[
1 \le s : =\max\{|f\cap g|:f,g\in E(F),\ f\ne g\} < k.
\]

\begin{theorem}[Resource stability]\label{thm:resource-semigroup}
For every fixed $2<q<\infty$, on $\mathcal G_\star$,
\begin{equation}\label{eq:resource-semigroup-goal}
 \|e^{-(r-1)t\mathsf T}\|_{L^q(\pi_k)\to L^q(\pi_k)}
 \le C_{F,\eta,q}(1+t)^{s+1}
\end{equation}
whenever $0\le t\le(r-1)^{-1}\log(1+(r-1)D_\star)$.
\end{theorem}

\medskip\noindent\textbf{Up and down averages.å}
For $0\le j\le k$, sample $e\sim\pi_k$ and then a uniformly chosen
$j$-set $J\subset e$.  The law of $J$ is
\begin{equation}\label{eq:shadow-pushforward}
 \pi_j(J)=\frac{1}{\binom kj}\sum_{e\supset J}\pi_k(e).
\end{equation}
Strict $k$-balance applied to two role edges meeting in $s$ vertices
implies
\begin{equation}\label{eq:shadow-density-gap}
 \rho(k-s)>1.
\end{equation}
Thus the one-edge handoff count and \eqref{eq:resource-degrees} give
\begin{equation}\label{eq:uniform-shadow-law}
 \pi_s(S)=\frac{1+O(n^{-c})}{\binom ns}
 \qquad(S\in\tbinom{[n]}s).
\end{equation}
The same pointwise comparison holds after conditioning on $J\subset S$.
The \emph{up} operator $\mathsf U_j^{(k)}$ lifts a function on
$j$-sets to resource edges by averaging over their $j$-subsets.
The \emph{down} operator $\mathsf D_j^{(k)}$ averages a function
on resources over those containing a prescribed $j$-set:
\begin{equation}\label{eq:resource-up-down}
 \begin{aligned}
 (\mathsf U_j^{(k)}g)(e)
   &=\frac1{\binom kj}\sum_{J\in\binom ej}g(J),\\
 (\mathsf D_j^{(k)}h)(J)
   &=\frac{\sum_{e\supset J}\pi_k(e)h(e)}
           {\sum_{e\supset J}\pi_k(e)}.
 \end{aligned}
\end{equation}
Values at $J$ outside the support of $\pi_j$ are immaterial.
These are adjoint conditional expectations for the indicated
measures, and are contractions on every $L^q$.
Thus $\mathsf U_j^{(k)}\mathsf D_j^{(k)}$ remembers a uniformly
chosen $j$-subset of the current resource and resamples a
resource containing it.  Abbreviate $\mathsf U=\mathsf U_s^{(k)}$
and $\mathsf D=\mathsf D_s^{(k)}$ for the largest intersection size $s$.

The probability that two ordered distinct role edges of $F$
meet in $j$ vertices is
\[
 \lambda_j=\frac{1}{r(r-1)}
 |\{(f,g)\in E(F)^2:f\ne g,\ |f\cap g|=j\}|.
\]
We have $\sum_{j=0}^s\lambda_j=1$.
For example, if $F=K_5^{(3)}$, each triple meets six of the
other nine triples in two vertices and three in one vertex.
Hence $\lambda_2=2/3$, $\lambda_1=1/3$, and the resource walk
is approximated by
\[
 \tfrac23\mathsf U_2^{(3)}\mathsf D_2^{(3)}
 +\tfrac13\mathsf U_1^{(3)}\mathsf D_1^{(3)}.
\]
The two terms represent the two possible amounts of retained
information in a resource transition.

Set
\[
 \mathsf T_0=\sum_{j=0}^s\lambda_j
              \mathsf U_j^{(k)}\mathsf D_j^{(k)}.
\]
We will show below that
\begin{equation}\tag{A}\label{eq:resource-walk-approx}
 \|\mathsf T-\mathsf T_0\|_{L^2(\pi_k)\to L^2(\pi_k)}\le n^{-c}.
\end{equation}

\medskip\noindent\textbf{The graph case from (A).}
Suppose $k=2$, so $s=1$.
Let $\Pi_1$ average under $\pi_1$.  Averaging over the empty
set is global averaging, so
\[
 \mathsf T_0=\mathsf U(\lambda_1I+\lambda_0\Pi_1)\mathsf D.
\]
On vertices, $\mathsf D\mathsf U=\tfrac12(I+\mathsf P)$,
where $\mathsf P$ selects a containing resource edge with weight
$\pi_k(e)$ and moves to its other endpoint.  It is a Markov
operator on $L^q(\pi_1)$, and hence
\[
 \|e^{-u\mathsf D\mathsf U}\|_{q\to q}
 \le e^{-u/2}\sum_{m\ge0}\frac{(u/2)^m}{m!}
      \|\mathsf P^m\|_{q\to q}\le1.
\]
Since $\mathsf D\mathsf U\Pi_1=\Pi_1\mathsf D\mathsf U=\Pi_1$,
\[
 e^{-u(\lambda_1I+\lambda_0\Pi_1)\mathsf D\mathsf U}
 =e^{-u\lambda_1\mathsf D\mathsf U}
   \bigl(I+(e^{-u\lambda_0}-1)\Pi_1\bigr),
 \qquad
 \|e^{-u(\lambda_1I+\lambda_0\Pi_1)\mathsf D\mathsf U}\|_{q\to q}\le2.
\]
For any bounded $A,B$, their power series give
\begin{equation}\label{eq:ABBA-identity}
 e^{-uAB}=I-A\int_0^u e^{-vBA}B\,dv.
\end{equation}
Apply this with $A=\mathsf U$ and
$B=(\lambda_1I+\lambda_0\Pi_1)\mathsf D$.
Both are contractions, so
$\|e^{-u\mathsf T_0}\|_{q\to q}\le1+2u$.

We will use the following perturbation estimate here and below.
If $\|e^{-uA}\|\le P(u)$ for a nondecreasing polynomial $P\ge1$
and $\|A-B\|\le\varepsilon$, Duhamel's formula gives
$f(u):=\|e^{-uB}\|\le P(u)+\varepsilon\int_0^uP(u-v)f(v)\,dv$.
Dividing by $P(u)$ and applying Gronwall yields
\begin{equation}\label{eq:polynomial-perturb}
 \|e^{-uB}\|
 \le P(u)\exp\left(\varepsilon\int_0^uP(v)\,dv\right).
\end{equation}
By (A), $\mathsf T-\mathsf T_0$ has $L^2$ norm $n^{-c}$.
Both are Markov operators; their $L^\infty$ norms are at most one,
so interpolation gives
$\|\mathsf T-\mathsf T_0\|_{q\to q}\le n^{-c_q}$
for each fixed $q>2$.  Apply \eqref{eq:polynomial-perturb}
with $P(u)=1+2u$.  Since $D_\star=n^{b\eta+o(1)}$,
the exponential factor is bounded for $u=O(\log D_\star)$.
Taking $u=(r-1)t$ proves the graph case of the theorem,
assuming (A).

\medskip\noindent\textbf{The shadow and stationary inputs.}
For general $F$, we need some further arguments.
For $j\le h\le k$, use the induced law $\pi_h$ from
\eqref{eq:shadow-pushforward} to define $\mathsf U_j^{(h)}$
and $\mathsf D_j^{(h)}$.  Conditional sampling in stages gives
\begin{equation}\label{eq:nested-conditional-operators}
 \mathsf U_j^{(k)}=\mathsf U\mathsf U_j^{(s)},
 \qquad
 \mathsf D_j^{(k)}=\mathsf D_j^{(s)}\mathsf D
 \quad(j\le s).
\end{equation}
Put
\begin{equation}\label{eq:resource-factorisation}
 \mathsf R_s=\sum_{j=0}^s\lambda_j
               \mathsf U_j^{(s)}\mathsf D_j^{(s)},
 \qquad \mathsf T_0=\mathsf U\mathsf R_s\mathsf D.
\end{equation}
Here $\mathsf D\mathsf U$ is the shadow walk: from an $s$-set
it chooses a containing surviving resource with its conditional
weight and then a uniform $s$-subset of that resource.
All these operators are Markov contractions.

As in the link-walk argument of Subsection~\ref{subsec:mixing},
we compare the actual conditional steps with ones sampled from
stationary laws.  Here the reference laws come from the uniform
incidence measure on the complete $k$-graph; we mark the resulting
operators with an overline.  On $\Omega_s=\binom{[n]}s$ define
\begin{equation}\label{eq:stationary-shadow-operators}
\begin{aligned}
 (\overline{\mathsf R}_s h)(S)
 &=\sum_{j=0}^s\frac{\lambda_j}{\binom sj\binom{n-j}{s-j}}
   \sum_{J\in\binom Sj}\sum_{\substack{T\in\Omega_s\\J\subset T}}h(T),\\
 (\overline{\mathsf D\mathsf U}h)(S)
 &=\frac{1}{\binom{n-s}{k-s}\binom ks}
   \sum_{\substack{e\in\binom{[n]}k\\S\subset e}}
   \sum_{T\in\binom es}h(T).
\end{aligned}
\end{equation}
Thus $\overline{\mathsf R}_s$ resamples from the uniform law
conditional on a chosen $j$-subset, while
$\overline{\mathsf D\mathsf U}$ uses uniform complete-graph
incidence sampling.  The latter still remembers the intermediate
$k$-edge; it is not a projection onto constants. We will show below that

\begin{equation}\tag{B}\label{eq:shadow-compare-goal}
 \|\mathsf D\mathsf U-\overline{\mathsf D\mathsf U}\|_{L^2(\pi_s)\to L^2(\pi_s)}
 \le n^{-c},
\end{equation}
\begin{equation}\tag{C}\label{eq:complete-semigroup-bound}
 \|e^{-u\overline{\mathsf R}_s\,\overline{\mathsf D\mathsf U}}\|_{L^q\to L^q}
 \le\sum_{h=0}^s\frac{(2u)^h}{h!}\le(1+2u)^s
 \qquad(1\le q\le\infty,\ u\ge0).
\end{equation}

\begin{proof}[Proof of Theorem~\ref{thm:resource-semigroup} for $k\ge3$, assuming (A)--(C)]
By \eqref{eq:uniform-shadow-law}, conditional resampling on the
$s$-slice is pointwise $1+O(n^{-c})$ close to its stationary reference.
Jensen's inequality, averaged over the random $j$-set, therefore gives
\begin{equation}\label{eq:shadow-R-compare}
 \|\mathsf R_s-\overline{\mathsf R}_s\|_{2\to2}\le n^{-c}.
\end{equation}
The operators in this comparison and in (B) are Markov
contractions for their respective measures, and the measures
are pointwise equivalent up to $1+O(n^{-c})$.
Their differences have bounded $L^\infty$ norm.
Interpolation thus gives, for each fixed $q>2$,
\begin{equation}\label{eq:shadow-product-perturb}
 \|\mathsf R_s\mathsf D\mathsf U-\overline{\mathsf R}_s\,\overline{\mathsf D\mathsf U}\|_{q\to q}
 \le\varepsilon,\qquad \varepsilon=n^{-c_{F,\eta,q}},
\end{equation}
where the norms may be taken under $\pi_s$ or uniform measure.
By (C), \eqref{eq:polynomial-perturb} with
$P(u)=(1+2u)^s$ gives, for every fixed $C$ and all large $n$,
\begin{equation}\label{eq:perturbation-semigroup-bound}
 \|e^{-u\mathsf R_s\mathsf D\mathsf U}\|_{q\to q}
 \le2(1+2u)^s\qquad(0\le u\le C\log D_\star).
\end{equation}
Apply \eqref{eq:ABBA-identity} with
$A=\mathsf U$ and $B=\mathsf R_s\mathsf D$.
By \eqref{eq:resource-factorisation}, contraction of $A,B$,
and the preceding bound,
\begin{equation}\label{eq:ab-ba-reduction}
 \|e^{-u\mathsf T_0}\|_{q\to q}
 \le1+\int_0^u\|e^{-v\mathsf R_s\mathsf D\mathsf U}\|_{q\to q}\,dv
 \le C_F(1+u)^{s+1}.
\end{equation}
Finally, (A), interpolation and \eqref{eq:polynomial-perturb}
transfer this bound to $e^{-u\mathsf T}$ throughout
$u=O(\log D_\star)$.  Taking $u=(r-1)t$ proves
\eqref{eq:resource-semigroup-goal}.
\end{proof}

\subsection{Resource approximation: reduction to fibre mixing}

We prove (A) in this subsection, assuming a lemma proved in a later subsection.
For distinct resources $e,f$, let $c(e,f)$ be the number of copies
of $F$ containing both.  Fix $j$ with $\lambda_j>0$ and a
$j$-subset $J$ of $e$.  The total weight of transitions whose
intersection is $J$ is
\[
 s_J(e)=\sum_{\substack{f\in E(H_\star)\\ e\cap f=J}}c(e,f).
\]

\begin{lemma}[Local role frequencies]\label{lem:role-counts}
Uniformly over $e$, $j$ with $\lambda_j>0$, and $J\in\binom ej$,
\begin{equation}\label{eq:role-counting}
 s_J(e)=(1+O(n^{-c}))
 \frac{(r-1)\lambda_j}{\binom kj}\,d_e.
\end{equation}
\end{lemma}

\begin{proof}
Count ordered pairs $(A,f')$ with $e,f'\in A$, $f'\ne e$, and
$e\cap f'=J$.  For each ordered role pair $(f,g)$ with
$|f\cap g|=j$, precisely $j!(k-j)!=k!/\binom kj$
bijections $\sigma:f\to e$ map $f\cap g$ onto $J$.  Hence
\[
 s_J(e)=\frac1{\operatorname{aut}(F)}
 \sum_{\substack{(f,g)\in E(F)^2\\f\ne g,\ |f\cap g|=j}}
 \sum_{\substack{\sigma:f\overset\sim\longrightarrow e\\
                 \sigma(f\cap g)=J}}\Phi_{F,\sigma}.
\]
By the rooted copy count \eqref{eq:rooted-F-count}, all the
summands have the same estimate.  There are $r(r-1)\lambda_j$
role pairs.  Dividing by the resource degree estimate
\eqref{eq:resource-degrees} proves the lemma.
\end{proof}

For $J$ with $\pi_j(J)>0$, set
$\mathcal R_J=\{e\in E(H_\star):J\subset e\}$.  On this fibre put
\[
 K_J(e,f)=
 \begin{cases}c(e,f)/s_J(e),&e\cap f=J,\\0,&\text{otherwise},\end{cases}
 \qquad
 \sigma_J(e)=\frac{s_J(e)}{\sum_{g\supset J}s_J(g)},
 \qquad
 \tau_J(e)=\frac{\pi_k(e)}{\sum_{g\supset J}\pi_k(g)}.
\]
Here $K_J$ is Markov and reversible with stationary law
$\sigma_J$, by $c(e,f)=c(f,e)$.  From \eqref{eq:role-counting}
we have $\sigma_J=(1+O(n^{-c}))\tau_J$ pointwise.
The only nonlocal assertion required is the following.

\begin{lemma}[Resource-fibre mixing]\label{lem:fibre-mixing}
Uniformly over $j$ with $\lambda_j>0$ and $j$-sets $J$,
\[
 \|K_J-\Pi_J\|_{L^2(\sigma_J)\to L^2(\sigma_J)}\le n^{-c},
 \qquad \Pi_Jh=\mathbb E_{\sigma_J}h.
\]
\end{lemma}

We deduce (A) from this lemma, proved in
Subsection~\ref{subsec:F-chain-counts}.

\begin{proof}[Proof of (A), assuming Lemma~\ref{lem:fibre-mixing}]
Write $\mathsf T_j$ for the part of $\mathsf T$ moving between
resources with intersection size $j$.  Exactly,
\begin{equation}\label{eq:exact-intersection-decomposition}
 (\mathsf T_jh)(e)=\sum_{J\in\binom ej}
   \frac{s_J(e)}{(r-1)d_e}(K_Jh)(e).
\end{equation}
Put $\mathsf K_jh(e)=\binom kj^{-1}\sum_{J\in\binom ej}K_Jh(e)$.
Lemma~\ref{lem:role-counts}, Jensen's inequality and the
comparison $\sigma_J=(1+O(n^{-c}))\tau_J$ give
$\|\mathsf T_j-\lambda_j\mathsf K_j\|_{2\to2}\le n^{-c}$.
Indeed, after sampling $e\sim\pi_k$ and then a uniform
$J\subset e$ of size $j$, the conditional law of $e$ is $\tau_J$;
this converts the averaged fibre $L^2$ norms into the global norm.

In the same sampling, $\mathsf U_j^{(k)}\mathsf D_j^{(k)}h(e)$
is the average over $J\in\binom ej$ of $\mathbb E_{\tau_J}h$.
Lemma~\ref{lem:fibre-mixing}, the equivalence of $\sigma_J$
and $\tau_J$, and Jensen therefore give
\[
\begin{aligned}
 &\|\mathsf K_jh-\mathsf U_j^{(k)}\mathsf D_j^{(k)}h\|_{L^2(\pi_k)}^2\\
 &\quad\le\sum_J\pi_j(J)
   \|K_Jh-\mathbb E_{\tau_J}h\|_{L^2(\tau_J)}^2
 \le n^{-2c}\sum_J\pi_j(J)\|h\|_{L^2(\tau_J)}^2
 =n^{-2c}\|h\|_{L^2(\pi_k)}^2.
\end{aligned}
\]
Summing over the finitely many occurring values of $j$ proves
\eqref{eq:resource-walk-approx}.
\end{proof}

\subsection{Shadow comparison: reduction to fibre mixing}

We prove (B) in this subsection, assuming a lemma proved in a later subsection.
Recall the maximum intersection size $s$ defined above.
The shadow walk $\mathsf D\mathsf U$ chooses a surviving $k$-edge
containing the current $s$-set, with weight proportional to
its resource degree, and then a uniform $s$-subset of that edge.
Call $0\le\ell<s$ \emph{feasible} if $2s-\ell\le k$.
Fix a feasible $\ell$ and an $\ell$-set $J$.
From an $s$-set $S\supset J$ the event that the next $s$-set $T$
satisfies $S\cap T=J$ has probability
$\binom{k-s}{s-\ell}/\binom ks$, independent of the selected resource.
Conditional on this event, the transition kernel is
\begin{equation}\label{eq:lower-shadow-kernel}
 \mathsf L_J^{(\ell)}(S,T)=
 \frac{\mathbf1_{\{S\cap T=J\}}}{\binom{k-s}{s-\ell}}
 \frac{\sum_{e\supset S\cup T}\pi_k(e)}
      {\sum_{e\supset S}\pi_k(e)}.
\end{equation}
It is Markov and reversible under
$\nu_J:=\pi_s(\cdot\mid J\subset S)$: multiplying
\eqref{eq:lower-shadow-kernel} by $\pi_s(S)$ produces an
expression symmetric in $S,T$.  Write
$\Pi_J^{(s)}h=\mathbb E_{\nu_J}h$.

\begin{lemma}[Shadow-fibre mixing]\label{lem:shadow-fibre}
Uniformly over feasible $\ell<s$ and $\ell$-sets $J$,
\[
 \|\mathsf L_J^{(\ell)}-\Pi_J^{(s)}\|_{L^2(\nu_J)\to L^2(\nu_J)}
 \le n^{-c}.
\]
\end{lemma}

\begin{proof}[Proof of (B), assuming Lemma~\ref{lem:shadow-fibre}]
Put
$\beta_\ell=\binom s\ell\binom{k-s}{s-\ell}/\binom ks$.
A step of $\mathsf D\mathsf U$ chooses $\ell$ with probability
$\beta_\ell$, then a uniformly chosen $\ell$-subset $J$ of $S$,
then follows $\mathsf L_J^{(\ell)}$.  When $\ell=s$ the step is
the identity, with mass $1/\binom ks$ in both walks.
For $\ell<s$, replace $\mathsf L_J^{(\ell)}$ by
$\Pi_J^{(s)}$.  Sampling $S\sim\pi_s$ and then $J$
gives conditional law $S\sim\nu_J$.  Jensen's inequality
and Lemma~\ref{lem:shadow-fibre} bound the change in $L^2$
by $n^{-c}$.

For the complete $k$-graph, the analogous exact-$J$ kernel is
uniform on $T$ with $S\cap T=J$, since the number of
containing $k$-edges depends only on $|S\cup T|=2s-\ell$.
This law differs in row total variation by $O_k(n^{-1})$
from uniform on all $s$-sets containing $J$.  Both kernels are
reversible for that uniform law, so duality and interpolation
also give an $O_k(n^{-1})$ bound in $L^2$.
By \eqref{eq:uniform-shadow-law}, $\nu_J$ is pointwise
$1+O(n^{-c})$ times this uniform law.  Applying the same
conditional Jensen bound compares the replaced transitions
with $\overline{\mathsf D\mathsf U}$ in $L^2(\pi_s)$ by $n^{-c}$.
The triangle inequality proves \eqref{eq:shadow-compare-goal}.
\end{proof}

\subsection{The elementary stationary bound}

\begin{proof}[Proof of (C)]
Call a matrix on $\Omega_s$ \emph{intersection-invariant} if its
$(S,T)$ entry depends only on $|S\cap T|$.  Such matrices form a
vector space of dimension at most $s+1$.  Their products are
intersection-invariant, since permuting $[n]$ merely relabels
the summation variable.  They are also symmetric; hence their
products are symmetric, so any two such matrices commute.

Both $\overline{\mathsf R}_s$ and
$\overline{\mathsf D\mathsf U}$ are intersection-invariant,
Markov and positive semidefinite on the uniform slice: each is
a conditional averaging map followed by its adjoint, or a
convex combination of such maps.  Thus they commute, and
$H:=\overline{\mathsf R}_s\,\overline{\mathsf D\mathsf U}$
is symmetric, positive semidefinite and Markov.  Its eigenvalues
therefore lie in $[0,1]$.

Every power of $H$ is intersection-invariant, so $H$ has at most
$s+1$ distinct eigenvalues $\theta_0,\ldots,\theta_{d-1}$.
Newton interpolation of $f(x)=e^{-ux}$ at these nodes gives
\[
 e^{-uH}
 =\sum_{h=0}^{d-1}f[\theta_0,\ldots,\theta_h]
   \prod_{i=0}^{h-1}(H-\theta_i I).
\]
The divided-difference mean-value theorem gives
$|f[\theta_0,\ldots,\theta_h]|\le u^h/h!$ on $[0,1]$.
Since $H$ is Markov, $\|H-\theta_i I\|_{q\to q}\le2$.
Taking norms yields
\[
 \|e^{-uH}\|_{q\to q}
 \le\sum_{h=0}^{d-1}\frac{(2u)^h}{h!}
 \le\sum_{h=0}^s\frac{(2u)^h}{h!}
 \le(1+2u)^s,
\]
which is (C).
\end{proof}

\subsection{Counting chains of copies of \texorpdfstring{$F$}{F}}
\label{subsec:F-chain-counts}

We now prove Lemma~\ref{lem:fibre-mixing}.  A step of $K_J$
uses an ordered pair of role edges meeting in the prescribed $j$-set
$J$.  To study a fixed power of $K_J$, we join several such pairs
using fresh copies of $F$.
An \emph{$F$-chain} of length $L$ consists of copies
$F_1,\ldots,F_L$ and $k$-edges $e_0,\ldots,e_L$ with
$e_{i-1},e_i\in E(F_i)$ and
\begin{equation}\label{eq:F-chain-fresh}
 V(F_i)\cap\left(e_0\cup\bigcup_{h<i}V(F_h)\right)=e_{i-1}
 \qquad(1\le i\le L).
\end{equation}
Put $W_i=V(F_i)\setminus e_{i-1}$.  Thus the $W_i$ are disjoint,
each of size $a=m-k$, and each copy adds exactly $b=r-1$
new edges.  An $F$-chain \emph{retains $J$} if
$e_{i-1}\cap e_i=J$ for every $i$.  Its endpoint root is
$I=e_0\cup e_L$.  For such a chain with $L\ge2$,
\begin{equation}\label{eq:excess-parameter}
 c_j=\rho(k-j)-1>0,
\end{equation}
and, since the only edges of the chain contained in $I$ are
$e_0$ and $e_L$, its number of free vertices and tested edges is
\begin{equation}\label{eq:chain-parameters}
 X_L=|V(C_L)\setminus I|=La-(k-j),\qquad
 Q_L=e(C_L)-e(C_L[I])=Lb-1=\rho X_L+c_j.
\end{equation}

We first verify the balance needed to apply the handoff counts.

\begin{lemma}[One-ended chain balance]\label{lem:one-ended-balance}
For a role edge $f\in E(F)$ and $f\subseteq U\subseteq V(F)$,
\begin{equation}\label{eq:original-template-balance}
 e(F[U])-1\le\rho(|U|-k).
\end{equation}
The inequality is strict when $U\ne V(F)$ and $F[U]$
contains another edge.  Consequently, for any $F$-chain $C$
rooted at $e_0$ and $U\supseteq e_0$,
\begin{equation}\label{eq:persistent-chain-balance}
 e(C[U])-1\le\rho(|U|-k).
\end{equation}
The latter inequality also holds for the reversed chain.
\end{lemma}

\begin{proof}
If $F[U]$ contains only $f$ the first inequality is immediate;
otherwise it follows from strict $k$-balance for a proper $U$,
and is an equality for $U=V(F)$.  In a chain, put
$s_i=|U\cap W_i|$ and let $q_i$ be the number of edges
of $F_i$ other than $e_{i-1}$ retained in $C[U]$.
Adjoining the entire incoming edge $e_{i-1}$ can only increase
this number.  The first inequality applied to that block gives
$q_i\le\rho s_i$.  Sum over blocks, noting that their newly
introduced edges are distinct.  The reverse argument is identical.
\end{proof}

\begin{lemma}[Endpoint-rooted chain counts]
\label{lem:persistent-chain-count}
For each $j$ with $\lambda_j>0$, sufficiently long $F$-chains
retaining $J$ are strictly balanced with root $I=e_0\cup e_L$,
uniformly over their ordered role sequences.  For sufficiently
small fixed $\eta$ there is a fixed length
$L_j=O_F(\eta^{-1})$ such that all these chains have at most
$\eta^{-4}$ vertices and endpoint-rooted trajectory at least $n$.
Consequently their extension counts, for all prescribed endpoints,
have relative error $n^{-c}$ on $\mathcal G_\star$.
\end{lemma}

\begin{proof}
Fix $I\subsetneq U\subsetneq V(C_L)$.
If $U$ meets some intermediate edge only in $J$, the reversed suffix
supplies the full excess $c_j$.  Otherwise every incomplete block
pays a positive deficit proportional to its missing vertices, which
beats that fixed excess when the chain is sufficiently long.

With $s_i=|U\cap W_i|$, let $q_i$ count the edges of
$C_L[U]$ in $E(F_i)\setminus\{e_{i-1}\}$ and put
\[
 \delta_i=\rho s_i-q_i\ge0,\qquad
 \Delta=\sum_i\delta_i,\qquad
 x=|U|-|I|,\qquad y=X_L-x=La-\sum_i s_i.
\]
As $e(C_L[I])=2$, we have
\[
 e(C_L[U])-2=\sum_iq_i-1=\rho x+c_j-\Delta.
\]
By \eqref{eq:chain-parameters}, strict rooted balance for this
$U$ is therefore equivalent to
\begin{equation}\label{eq:reuse-surplus}
 \Delta X_L>c_jy.
\end{equation}

If $U\cap e_h=J$ for some $1\le h<L$, apply
Lemma~\ref{lem:one-ended-balance} to the \emph{reversed suffix}
rooted at $e_L$.  Its vertices retained outside $e_L$ number
$j+\sum_{i>h}s_i-k$, and its tested edges retained outside $e_L$
number $\sum_{i>h}q_i-1$.  Thus
$\sum_{i>h}\delta_i\ge\rho(k-j)-1=c_j$.
Since $x>0$ implies $y<X_L$, this proves
\eqref{eq:reuse-surplus}.

Otherwise $U\cap e_h\supsetneq J$ for every $1\le h<L$.
Whenever $s_i<a$ we have $s_i>0$: for $i<L$ this follows from
the presence of a vertex of $e_i\setminus J$, and for $i=L$
from $e_L\subset I$.  If $q_i=0$, then $\delta_i=\rho s_i>0$.
If $q_i>0$, adjoining the incoming edge and applying the strict
part of \eqref{eq:original-template-balance} gives $\delta_i>0$.
There are only finitely many local role pairs and vertex subsets,
so some $\gamma_{F,j}>0$ satisfies
$\delta_i\ge\gamma_{F,j}(a-s_i)$ whenever $s_i<a$.
Hence $\Delta\ge\gamma_{F,j}y$, which gives
\eqref{eq:reuse-surplus} as soon as $X_L>c_j/\gamma_{F,j}$.
This threshold is uniform over role sequences.  If $U=I$
there are no free vertices, and if $U=V(C_L)$ we have equality,
as required.

Finally, at $p_\star=n^{-1/\rho+\eta}$,
\begin{equation}\label{eq:exponent-chain-weight}
 \log_n\widehat\phi_{C_L,I}(i^\star)
 =-c_j/\rho+\eta(Lb-1)+o(1).
\end{equation}
Choose a fixed $L_j$ above the balance threshold so that this
exponent exceeds $1$ by a positive constant.  Then
$L_j=O_F(\eta^{-1})$ and $|V(C_{L_j})|\le\eta^{-4}$
for small $\eta$.  Lemma~\ref{lem:JK-cutoff-adapter} and the
balanced-template count \eqref{eq:JK-balanced-count} apply
uniformly to these finitely many decorated templates and all their
root injections.
\end{proof}

\begin{proof}[Proof of Lemma~\ref{lem:fibre-mixing}]
Fix $j,J$ and the fixed length $L=L_j$ from
Lemma~\ref{lem:persistent-chain-count}; put $m_J=|\mathcal R_J|$.
Lemma~\ref{lem:role-counts} and \eqref{eq:resource-degrees} give
\begin{equation}\label{eq:resource-fibre-uniform}
 \sigma_J(e)=\frac{1+O(n^{-c})}{m_J}
 \qquad(e\in\mathcal R_J).
\end{equation}
The one-edge rooted count at $J$ gives
$m_J=(1+O(n^{-c}))\binom{n-j}{k-j}p_\star$;
its trajectory is polynomially large since $c_j>0$.
For $e\supset J$ and $x\in e\setminus J$, the enlarged-root
bound \eqref{eq:JK-enlarged-root-count} gives
\[
 |\{f\in\mathcal R_J:x\in f\}|
 \le(1+\log n)^{O_F(1)}\max\{1,n^{k-j-1}p_\star\}.
\]
Indeed the least trajectory after also prescribing $x$ is
$\min\{n,n^{k-j}p_\star\}$.  Dividing by $m_J$ shows that
\begin{equation}\label{eq:nongeneric-intersection}
 \frac{|\{f\in\mathcal R_J:e\cap f\supsetneq J\}|}{m_J}
 \le (1+\log n)^{O_F(1)}
 \bigl(n^{-1}+(n^{k-j}p_\star)^{-1}\bigr)\le n^{-c}.
\end{equation}
Thus $G_e=\{f\in\mathcal R_J:e\cap f=J\}$ has
$|G_e|=(1+O(n^{-c}))m_J$.

Expand $K_J^L(e,f)$ over its successive copies of $F$.
A nongeneric trajectory has a first step at which a new copy
uses a previously exposed vertex outside its incoming role edge.
Prescribing that vertex costs a factor $n^{-c}$ in the
single-copy rooted extension count: for a proper rooted subtemplate
with tested edges this follows from strict $k$-balance, for one
with no new edge from its additional vertex, and for the full
copy from the factor $n^{b\eta}$.  The enlarged-root estimate
\eqref{eq:JK-enlarged-root-count}, normalised by the
unrestricted count $n^ap_\star^b$ and
Lemma~\ref{lem:role-counts}, therefore makes the conditional
probability of each first reuse at most $n^{-c}$.
As $L$ and the number of previously exposed vertices are bounded,
the total nongeneric row mass is $O(n^{-c})$.

For $f\in G_e$, label the $L$ copies of a generic chain,
choose their ordered incoming and outgoing role edges, and
identify consecutive role edges along the common $J$.
These choices give finitely many decorated templates, each with
$X_L$ free vertices and $Q_L$ tested edges.  The number of
compatible root bijections depends only on $k,j$, not on
$(e,f)$, and every generic ordered trajectory has the same
multiplicity $\operatorname{aut}(F)^L$ among these labelled
embeddings.  Lemma~\ref{lem:persistent-chain-count} estimates
each rooted count by $(1+O(n^{-c}))n^{X_L}p_\star^{Q_L}$.
It follows that the number of generic trajectories from $e$ to
$f$ is $(1+O(n^{-c}))a_{n,j,L}$, for a positive constant
$a_{n,j,L}$ independent of the endpoints.
By Lemma~\ref{lem:role-counts} and \eqref{eq:resource-degrees},
each step denominator is
$(1+O(n^{-c}))(r-1)\lambda_jD_\star/\binom kj$.
Hence the generic contribution to $K_J^L(e,f)$ is
$(1+O(n^{-c}))b_{n,j,L}$ on $G_e$ and zero elsewhere,
with $b_{n,j,L}>0$ independent of $e,f$.

Put $\varepsilon=n^{-c}$, decreasing $c$ if necessary.
Every row of $K_J^L$ has exceptional mass $O(\varepsilon)$;
its generic mass is $(1+O(\varepsilon))b_{n,j,L}$ at each
of $(1+O(\varepsilon))m_J$ good endpoints.
As the rows sum to one,
$b_{n,j,L}=(1+O(\varepsilon))/m_J$.  Together with
\eqref{eq:resource-fibre-uniform}, this gives
\[
 \sup_e\sum_f|K_J^L(e,f)-\sigma_J(f)|=O(\varepsilon).
\]
We deduce $\|K_J^L-\Pi_J\|_{\infty\to\infty}=O(\varepsilon)$
and so $\|K_J^L-\Pi_J\|_{2\to2}=O(\varepsilon)$ by self-adjointness.
\end{proof}

\subsection{Counting chains of surviving edges}
\label{subsec:shadow-chain-counts}

We prove Lemma~\ref{lem:shadow-fibre} by counting chains of
surviving $k$-edges rather than copies of $F$.

Fix a feasible $\ell<s$.  A \emph{shadow chain retaining $J$}
of length $L$ is constructed from mutually disjoint labelled sets
\[
 J,U_0,\ldots,U_L,P_1,\ldots,P_L,
 \qquad |J|=\ell,\quad |U_i|=s-\ell,
 \quad |P_i|=f:=k-2s+\ell\ge0,
\]
with named $s$-states and $k$-edges
\begin{equation}\label{eq:shadow-chain-definition}
 S_i=J\cup U_i\quad(0\le i\le L),\qquad
 e_i=J\cup U_{i-1}\cup U_i\cup P_i\quad(1\le i\le L).
\end{equation}
Its endpoint root is $I=S_0\cup S_L$, it has $L$ tested edges,
and its number of non-root vertices is
\begin{equation}\label{eq:shadow-chain-free}
 X_L=L(k-s)-(s-\ell).
\end{equation}
A \emph{generic shadow trajectory} is an injective,
edge-preserving embedding of this ordered template; we do not
require an induced embedding.  No vertices are identified beyond
the common $J$ and consecutive shared states.

\begin{lemma}[Shadow-chain balance]\label{lem:shadow-chain-balance}
For every feasible $\ell<s$ and $L\ge2$, the template
\eqref{eq:shadow-chain-definition} is strictly balanced when
rooted at $I$.  Every proper rooted subtemplate with at least
one tested edge has density at most $1/(k-s)$, whereas the
full rooted density is
\begin{equation}\label{eq:shadow-chain-strict-balance}
 \frac{L}{L(k-s)-(s-\ell)}>\frac1{k-s}.
\end{equation}
\end{lemma}

\begin{proof}
Suppose a proper rooted subtemplate retains $q>0$ of the
resource edges, forming $c$ intervals of positions.  Let $h$
be the number of these intervals that touch an endpoint of the
chain.  Here $h\le2$ and $c-h\ge0$, since an interval touching
both endpoints would contain every edge.  Those $q$ edges
force all $qf$ of their private vertices and at least
$(q+c-h)(s-\ell)$ vertices from the intermediate $U_i$'s
outside the prescribed root.  They therefore force at least
\[
 qf+(q+c-h)(s-\ell)
 =q(k-s)+(c-h)(s-\ell)\ge q(k-s)
\]
non-root vertices.  This bounds their density by $1/(k-s)$.
A subtemplate with no tested edges has density zero.
The full chain has $L$ tested edges and
\eqref{eq:shadow-chain-free} free vertices, giving
\eqref{eq:shadow-chain-strict-balance}.
\end{proof}

Choose a fixed $L\ge2$, uniformly for the finitely many
feasible $\ell$, so large that
\begin{equation}\label{eq:shadow-chain-length}
 L(k-s-1/\rho)>s-\ell.
\end{equation}
Such a choice exists by \eqref{eq:shadow-density-gap}.
Then the chain's terminal rooted trajectory is
\[
 \widehat\phi_{\mathcal H,I}(i^\star)
 =n^{L(k-s)-(s-\ell)}p_\star^L
 =n^{L(k-s-1/\rho)-(s-\ell)+L\eta},
\]
a fixed positive power of $n$.  Its order
$s+L(k-s)\le\eta^{-4}$ for sufficiently small $\eta$.
Lemma~\ref{lem:JK-cutoff-adapter} and
\eqref{eq:JK-balanced-count} therefore count its embeddings,
with relative error $n^{-c}$, uniformly over prescribed endpoint
injections.

\begin{proof}[Proof of Lemma~\ref{lem:shadow-fibre}]
Fix $J,\ell,L$ as above.  By \eqref{eq:uniform-shadow-law},
$\nu_J$ is pointwise $1+O(n^{-c})$ times uniform on the
$m_J=\binom{n-\ell}{s-\ell}$ states containing $J$.
For a starting state $S_0$, the possible generic endpoints are
\[
 G_{S_0}=\{S_L\supset J:S_0\cap S_L=J\}.
\]
They number $\binom{n-s}{s-\ell}=(1+O_s(n^{-1}))m_J$.

Expand $(\mathsf L_J^{(\ell)})^L(S_0,S_L)$ over
intermediate states and chosen resource edges.  A trajectory is
generic precisely when it realises
\eqref{eq:shadow-chain-definition} injectively.  At the first
nongeneric step, the selected $k$-edge $e\supset S$ contains
some previously exposed vertex $x\notin S$.
Conditioning on the prescribed overlap $S\cap T=J$ does not
change the choice of $e$, since each $e\supset S$ allows
$\binom{k-s}{s-\ell}$ choices of $T$.  The enlarged-root bound
\eqref{eq:JK-enlarged-root-count}, applied to a one-edge template
rooted at $S$ with the extra prescription $x$, gives
\[
 \Prb(x\in e\mid S,J)
 \le(1+\log n)^{O_F(1)}
 \bigl(n^{-1}+(n^{k-s}p_\star)^{-1}\bigr)\le n^{-c}.
\]
Indeed, the smallest enlarged-root trajectory is
$\min\{n,n^{k-s}p_\star\}$, and the denominator is
$\Theta_F(n^{k-s}p_\star)$ by the one-edge handoff count.
There are only $O_{k,s,L}(1)$ earlier vertices and steps, so
nongeneric trajectories have total row mass $n^{-c}$.

For fixed $S_0,S_L\in G_{S_0}$, every compatible injection of
the labelled endpoint root has the same number of choices,
independent of $S_0,S_L$; it simply orders the vertices of
$J$, $S_0\setminus J$ and $S_L\setminus J$.  Moreover, every
generic ordered trajectory corresponds to the same number of
labelled embeddings of \eqref{eq:shadow-chain-definition},
independent of its endpoints: the labels order the intermediate
$U_i$ and the private $P_i$.  The rooted chain count above gives
the common trajectory $n^{X_L}p_\star^L$, with relative error
$n^{-c}$.  Hence the number of generic ordered trajectories
from any $S_0$ to any $S_L\in G_{S_0}$ is
$(1+O(n^{-c}))a_{n,\ell,L}$ for one positive number
$a_{n,\ell,L}$ independent of the endpoints.

In the expansion of \eqref{eq:lower-shadow-kernel}, each step
has weight
\[
 \frac{d_e}{\binom{k-s}{s-\ell}\sum_{g\supset S}d_g}
 =\frac{1+O(n^{-c})}
 {\binom{k-s}{s-\ell}\binom{n-s}{k-s}p_\star},
\]
by \eqref{eq:resource-degrees} and
the one-edge handoff count.  The generic contribution to
the $L$-step endpoint transition is therefore
$(1+O(n^{-c}))b_{n,\ell,L}$ on $G_{S_0}$, independent of
$S_0,S_L$, and zero outside $G_{S_0}$.

The same row-sum argument as in the proof of
Lemma~\ref{lem:fibre-mixing} now makes the common generic
transition weight $(1+O(n^{-c}))/m_J$.
The exceptional mass is $O(n^{-c})$, so the rows of
$(\mathsf L_J^{(\ell)})^L$ differ from $\nu_J$ in total variation
by $O(n^{-c})$. This gives an $L^\infty$ bound, so by self-adjointness we deduce
\[
 \|\mathsf L_J^{(\ell)}-\Pi_J^{(s)}\|_{L^2(\nu_J)\to L^2(\nu_J)}
 \le n^{-c/L}.
\]
This proves Lemma~\ref{lem:shadow-fibre}, and hence completes
the proof of Theorem~\ref{thm:resource-semigroup}.
\end{proof}

\section{Short incidence cycles}
\label{sec:cycles}

It remains to justify the assumption that few candidates lie on short cycles.
Recall that $b=r-1$ and that $\chi_R(\Gamma_\star)$ is the fraction of
candidates on incidence cycles of length at most $R$.

\begin{theorem}[Short-cycle candidate bound]
\label{thm:cyclic-sparsity}
For each fixed integer $R\ge4$ there are constants
$C_{F,R}<\infty$ and $\eta_{F,R}>0$ such that, if
$0<\eta<\eta_{F,R}$, then uniformly on $\mathcal G_\star$,
\begin{equation}\label{eq:short-cycle-target}
 \chi_R(\Gamma_\star)
 \le n^{-1/b+C_{F,R}\eta}.
\end{equation}
\end{theorem}

A \emph{self-avoiding cyclic $F$-union} is the union
$S=F_1\cup\cdots\cup F_\ell$ of distinct copies of $F$, together with
distinct resource edges $e_1,\ldots,e_\ell$ satisfying
$e_i\in E(F_i)\cap E(F_{i+1})$ (cyclic indices).  

\begin{lemma}[Bounded cyclic-type count]
\label{lem:cyclic-type-count}
Let $S$ be a fixed self-avoiding cyclic $F$-union with
$|V(S)|\le\eta^{-4}$.  On $\mathcal G_\star$, the number of embeddings
of $S$ into $H_\star$ is at most
\[
 \#\{S\hookrightarrow H_\star\}
 \le n^{k-1/\rho-1/b+O_{F,S}(\eta)}.
\]
\end{lemma}

\begin{proof}[Proof of Theorem~\ref{thm:cyclic-sparsity}, assuming Lemma~\ref{lem:cyclic-type-count}]
Fix $R\ge4$.  A simple incidence cycle of length $2\ell\le R$
in $\Gamma_\star$ alternates between distinct candidates and distinct
resources.  Its candidates are therefore distinct copies
$F_1,\ldots,F_\ell$ of $F$, while its resources are distinct
$k$-edges $e_1,\ldots,e_\ell$ with
$e_i\in E(F_i)\cap E(F_{i+1})$ (cyclic indices).  Hence the
union $S=\bigcup_{i=1}^\ell F_i$ is exactly a self-avoiding cyclic
$F$-union in the sense defined above, including when $\ell=2$ or when
there are further overlaps.  It has at most $mR/2$ ordinary vertices.
There are finitely many isomorphism types of such unions, depending
only on $F,R$, and for each embedded union only $O_{F,R}(1)$ choices of
the ordered candidates, shared resources, incidence cycle and a candidate
on that cycle.  Consequently the number of candidates on short cycles satisfies
the explicit finite-type bound
\[
 M\chi_R(\Gamma_\star)
 \le O_{F,R}(1)
 \sum_{\substack{S\text{ a cyclic }F\text{-union type}\\
                    v(S)\le mR/2}}
 \#\{S\hookrightarrow H_\star\}.
\]
Take $\eta$ sufficiently small for the finite collection of union types.  Summing the bound of Lemma~\ref{lem:cyclic-type-count} over the
finitely many types bounds the number of candidates on cycles of
length at most $R$ by
\(
 n^{k-1/\rho-1/b+O_{F,R}(\eta)}.
\)
Corollary~\ref{cor:near-regular-handoff} gives
$M=n^{k-1/\rho+r\eta+o(1)}$; dividing by $M$ gives \eqref{eq:short-cycle-target}.
\end{proof}

\subsection{Estimates at the handoff}
\label{sec:handoff-inputs}

Here we record some simple consequences of the
rooted-template estimates of Joos and K\"uhn \cite{JoosKuhn}
used in Sections~\ref{sec:persistent-chains}--\ref{sec:cycles}.
Recall $a=m-k$, $b=r-1$ and $\rho=\rho_F=b/a$ from the
handoff reduction.  Rooted-template notation and the deterministic
trajectory $\widehat\phi_{A,I}$ were defined in
Subsection~\ref{subsec:removal-matching}.
Take the parameter of \cite{JoosKuhn} to be
$\varepsilon_{\rm JK}=2\eta^2$.  For sufficiently small fixed $\eta$,
their controlled trajectory reaches density
$n^{-1/\rho+2\eta^2}<p_\star$, and
$\varepsilon_{\rm JK}^{-4}>\eta^{-4}$.  Their Section~7 stopping event tracks all balanced and strictly
balanced rooted templates of order at most
$\varepsilon_{\rm JK}^{-4}$, so in particular those of order
at most $\eta^{-4}$.  Enlarged prescriptions are covered by
their Lemma~7.16.

For deterministic $i$ before termination define the error scale
\[
 \zeta(i):=\frac{n^{\eta^2}}
 {n^{1/2}\widehat p(i)^{\rho/2}}.
\]
At the handoff, $\zeta(i^\star)=n^{\eta^2-\rho\eta/2+o(1)}
=n^{-\Omega_F(\eta)}$ for sufficiently small fixed $\eta$.
For $\theta>0$, let
\[
 i_{A,I}^{\theta}
 :=\min\{i:\widehat\phi_{A,I}(i)
 \le\zeta(i)^{-\theta}\}.
\]

\begin{proposition}[Uniform handoff event for bounded templates]
\label{prop:JK-handoff}
Fix $F$ and sufficiently small $\eta>0$.  For every sufficiently small
$\delta_{\rm JK}>0$ depending only on $F,\eta$, there is an event
$\mathcal G_\star$ with
\[
 \Prb(\mathcal G_\star^c)
 \le\exp\{-(\log n)^{4/3}\}
\]
such that the following hold simultaneously for \emph{all} rooted
$k$-templates of order at most $\eta^{-4}$ satisfying the hypotheses
in the respective parts, and for \emph{all} admissible prescribed root
injections.  In parts (iii) and (iv), $C_{A,I}$ and $C_{A,J}$ denote
fixed template-dependent exponents.

\begin{enumerate}[label=\textup{(\roman*)}]
\item For every role edge $f\in E(F)$, every resource
$e\in E(H_\star)$, and every bijection $\sigma:f\to e$,
\begin{equation}\label{eq:rooted-F-count}
 \Phi_{F,\sigma}
 =\left(1+n^{-\Omega_{F,\eta}(1)}\right)n^ap_\star^b.
\end{equation}

\item Let $(A,I)$ be a balanced template with
$|V(A)|\le\eta^{-4}$.  If
\(
 i^\star<i_{A,I}^{\delta_{\rm JK}^{1/2}},
\)
then for all root embeddings $\psi$ we have
\begin{equation}\label{eq:JK-balanced-count}
 \Phi_{A,\psi}
 =\left(1+O(\zeta(i^\star)^{\delta_{\rm JK}})\right)
 \widehat\phi_{A,I}(i^\star).
\end{equation}

\item For every strictly balanced template $(A,I)$ with
$|V(A)|\le\eta^{-4}$ we have
\[
 \Phi_{A,\psi}
 \le(1+\log n)^{C_{A,I}}
 \max\{1,\widehat\phi_{A,I}(i^\star)\}.
\]

\item Let $I\subsetneq J\subseteq V(A)$ and
$|V(A)|\le\eta^{-4}$.  Among subtemplates $(B,I)\subseteq(A,I)$ whose vertex
set contains $J$, choose one minimising $\widehat\phi_{B,I}(i^\star)$.  Then
for all injections $\psi:J\hookrightarrow V(H_\star)$ we have
\begin{equation}\label{eq:JK-enlarged-root-count}
 \Phi_{A,\psi}
 \le(1+\log n)^{C_{A,J}}
 \frac{\widehat\phi_{A,I}(i^\star)}
 {\widehat\phi_{B,I}(i^\star)}.
\end{equation}
\end{enumerate}
\end{proposition}

\begin{proof}
Apply \cite[Sections~5--7, Lemmas~7.14 and~7.16, and Lemma~10.1]{JoosKuhn}
to the complete initial $k$-graph.  Every bounded rooted template has
initial count $(1+O(n^{-1}))\widehat\phi_{A,I}(0)$, so the common
stopping event controls all the required templates and root injections.
The choice $\varepsilon_{\rm JK}=2\eta^2<\eta$ places $i^\star$ before
the controlled horizon; their Lemma~10.1 gives the exceptional
probability.  Their balanced counts, strictly balanced upper bounds,
and enlarged-root bounds give (ii)--(iv), and the rooted copy-degree
estimate gives (i).  In (iv), prescribing additional vertices of $J$
does not change the tested-edge set $E(A)\setminus E(A[I])$.
\end{proof}

The balanced-template cutoff also yields the following consequence.

\begin{lemma}[Polynomial trajectory lies before the balanced cutoff]
\label{lem:JK-cutoff-adapter}
For fixed $F,\eta$ there is a choice of $\delta_{\rm JK}>0$,
depending only on $F,\eta$ and compatible with
Proposition~\ref{prop:JK-handoff}, such that the following holds
\emph{simultaneously}.  Let $(A,I)$ be any balanced rooted template
with at least one edge outside $A[I]$ and $|V(A)|\le\eta^{-4}$.
If, for some fixed $\xi>0$,
\begin{equation}\label{eq:JK-positive-trajectory}
 \widehat\phi_{A,I}(i^\star)\ge n^\xi,
\end{equation}
then
\begin{equation}\label{eq:JK-cutoff-check}
 i^\star<i_{A,I}^{\delta_{\rm JK}^{1/2}}.
\end{equation}
Consequently Proposition~\ref{prop:JK-handoff}(ii) gives a
polynomially small relative error for $\Phi_{A,\psi}$, uniformly
over all applicable templates and prescribed root injections.
\end{lemma}

\begin{proof}
Write $q=e(A)-e(A[I])\ge1$ and $t=|V(A)|-|I|$.
The terminal trajectory has exponent
\[
 \widehat\phi_{A,I}(i^\star)
 =n^{t-q/\rho+q\eta+o(1)}.
\]
Since the template order is at most $\eta^{-4}$, all positive
exponents $t-q/\rho+q\eta$ have a common positive lower bound.  Write $\delta=\delta_{\rm JK}$.
Since
\[
 \zeta(i^\star)^{-\delta^{1/2}}
 =n^{(\rho\eta/2-\eta^2)\delta^{1/2}+o(1)},
\]
we may choose $\delta$ once, sufficiently small in terms of $F,\eta$,
so that the exponent on the right is below half this common lower
bound and $q>\rho\delta^{1/2}/2$ holds for every $q\ge1$.
Up to a factor independent of $i$, the ratio
$\widehat\phi_{A,I}(i)/\zeta(i)^{-\delta^{1/2}}$
is then a positive power of $\widehat p(i)$; it decreases as the
process advances.  At $i^\star$ the ratio is bigger than one, and
therefore \eqref{eq:JK-cutoff-check} holds simultaneously for all the templates in the
statement.  Decreasing $\delta$ further if necessary also meets the
conditions of Proposition~\ref{prop:JK-handoff}.
\end{proof}

\subsection{Cyclic-template counts}

The counting argument uses the cyclic-union structure established by
Joos and K\"uhn \cite[Lemmas~12.2, 12.3 and~12.5]{JoosKuhn} and their
supercritical-extension argument \cite[Lemma~12.6]{JoosKuhn}.

\begin{proof}[Proof of Lemma~\ref{lem:cyclic-type-count}]
Let $S$ be a cyclic $F$-union. Lemma~12.5 in \cite{JoosKuhn} gives a
resource edge $e$ with $\rho_{S,e}>\rho=b/a$.  The integer surplus is
therefore
\[
 a(e(S)-1)-b(v(S)-k)\ge1,
 \qquad
 v(S)-e(S)/\rho\le k-1/\rho-1/b.
\]
The exponent $k-1/\rho-1/b$ is positive: strict balance excludes a
matching, so some two role edges meet and $\rho\ge1/(k-1)$, giving
$k-1/\rho-1/b\ge1-1/b>0$.

By \cite[Lemmas~12.2 and~12.3]{JoosKuhn}, either $(S,\varnothing)$
is strictly balanced or it contains a proper unrooted subtemplate $B$
that is strictly balanced with $e(B)/v(B)\ge\rho$, chosen with
maximum density and minimum vertex count.  In the second case every
extension root $J\supseteq V(B)$ has density $\rho_{S,J}\ge\rho$.

In the first case the strictly balanced handoff bound from
Proposition~\ref{prop:JK-handoff}(iii) gives
\[
 \#\{S\hookrightarrow H_\star\}
 \le(1+\log n)^{O_{F,S}(1)}n^{v(S)}p_\star^{e(S)}
 \le(1+\log n)^{O_{F,S}(1)}
     n^{k-1/\rho-1/b+O_{F,S}(\eta)}.
\]
In the second case the same handoff bound counts $B$ in at most
$(1+\log n)^{O_{F,S}(1)}n^{O_{F,S}(\eta)}$ ways.  
By \cite[Lemma~12.6]{JoosKuhn}, using the upper bound
in Proposition~\ref{prop:JK-handoff}(iii), gives at most another
$(1+\log n)^{O_{F,S}(1)}n^{O_{F,S}(\eta)}$ completions:
the density condition $\rho_{S,J}\ge\rho$ is retained when the root
is enlarged during that induction.  
\end{proof}

\bibliographystyle{plain}
\bibliography{refs}

\end{document}